\documentclass[a4paper,11pt]{amsart}
\usepackage[utf8]{inputenc}
\usepackage[english]{babel}
\usepackage{fancyhdr}
\usepackage{amsmath}
\usepackage{amsfonts,amstext}
\usepackage{amssymb,amsthm,amscd,amsxtra,wasysym,graphicx}
\usepackage{grffile}
\usepackage{mathrsfs,esint,comment}
\usepackage{tikz-cd}
\usepackage{enumitem,mathtools,dsfont}
\usepackage{palatino,mathpazo}
\usepackage{xcolor}

\def\Bibtex{{\rm B\kern-.05em{\sc i\kern-.025em b}\kern-0.08em T\kern-.1667em\lower.7ex\hbox{E}\kern-.125emX}}
\usepackage[breaklinks=true,bookmarks=false,pagebackref]{hyperref}
\hypersetup{linktocpage=true,
	unicode=false,          
	pdftoolbar=true,       
	pdfmenubar=true,       
	pdffitwindow=false,     
	pdfstartview={FitH}, 
	pdftitle={Monge-Amp\`ere equations},   
	pdfauthor={Quang-Tuan Dang},   
	colorlinks=true,  
	linkcolor=purple,         
	citecolor=blue,        
	filecolor=green,
	urlcolor=blue}          

\usepackage{geometry}
 \usepackage{fullpage}

\usepackage{cleveref}
\usepackage{indentfirst}
\theoremstyle{plain}
\newtheorem{theorem}{Theorem}
\newtheorem{lemma}[theorem]{Lemma}

\newtheorem{proposition}[theorem]{Proposition}

\newtheorem{question}[theorem]{Question}

\theoremstyle{definition}
\newtheorem{definition}[theorem]{Definition}
\newtheorem{remark}[theorem]{Remark}

\theoremstyle{plain}
\newtheorem{bigthmm}{Theorem}

\numberwithin{theorem}{section}
\numberwithin{equation}{section}
 
\DeclareMathOperator \PSH {{\rm PSH}}

\DeclareMathOperator \tr {\textrm{tr}}

\def\om{\omega}
\def\tom{\Tilde{\omega}}

\def\f{\varphi}
\def\dc{dd^c}

\def\p{\psi}

\def\tmax{T_{\max}}

\begin{document}
	
	\title[Singularities of the Chern--Ricci flow]{Singularities of the Chern--Ricci  flow}
	\author{Quang-Tuan Dang}
	\address{The Abdus Salam International Centre for Theoretical Physics (ICTP), Str. Costiera, 11, 34151 Trieste, TS, Italy}
	\email{qdang@ictp.it}
	\date{\today}

	\begin{abstract} We study the nature of finite time singularities for the Chern--Ricci flow, partially answering a question posed by Tosatti--Weinkove.
		We show that a solution of degenerate parabolic complex Monge--Amp\`ere equations starting from arbitrarily positive (1,1)-currents, is smooth outside some analytic subset, generalizing works by Di Nezza--Lu. Moreover, we extend Guedj--Lu's recent approach to establish uniform a priori estimates for degenerate complex Monge--Amp\`ere equations on compact Hermitian manifolds.
		We apply these results to study the Chern--Ricci flow on log terminal varieties starting from a current with mild singularities.
	\end{abstract}
	\maketitle
	\tableofcontents

	\section{Introduction}\label{sect: intro}
	
	Finding canonical metrics on complex varieties has been a central problem in
	complex geometry over the last few decades. Since Yau's solution to Calabi's
	conjecture, significant progress has been made in this direction. Cao~\cite{cao1985deformation} introduced a parabolic approach to provide an alternative proof of the existence of K\"ahler--Einstein metrics on manifolds with numerically trivial or ample canonical line bundle via the K\"ahler--Ricci flow. This flow is only Hamilton's Ricci flow evolving K\"{a}hler metrics.
	Motivated by the classification of complex varieties, Song--Tian~\cite{song2012canonical,song2017kahler} have proposed an {\it Analytic Minimal Model Program} to classify algebraic varieties with mild singularities using the K\"{a}hler--Ricci flow. This approach necessitates a theory of weak solutions for degenerate parabolic complex Monge--Ampère equations starting from rough initial data. 
	Since then, various results have been achieved in this direction. Song and Tian initiated the study of the K\"ahler--Ricci flow starting from an initial current with
	continuous potentials. While Guedj--Zeriahi~\cite{guedj2017regularizing} (also~\cite{to2017regularizing}) showed that
	the K\"ahler--Ricci flow could be continued from an initial current with zero Lelong numbers.
	To the author's knowledge, the best results so far have been obtained by Di Nezza--Lu~\cite{di2017uniqueness}, who successfully ran the K\"ahler--Ricci flow from an initial current with positive Lelong numbers. There have been several related works in such singular settings from a pluripotential theoretical point of view. For further details, we refer to the recent works~\cite{guedj2020pluripotential,dang2021pluripotential}
	and the references therein.  
	
	\medskip
	
	Beyond the K\"ahler setting, there more recently has been interest in the study of geometric flows in the context of
	non-K\"ahler manifolds. 
	Unlike the K\"ahler case, Hamilton's Ricci flow does not, in
	general, preserve the special Hermitian condition. It is thus natural to look for another geometric flow of Hermitian metrics, which somehow specializes in the Ricci flow in the K{\"a}hler context. Several parabolic flows on complex manifolds that preserve the Hermitian property have been proposed by Streets--Tian~\cite{streets2011hermitian,streets2010parabolic} and Liu--Yang~\cite{liu2012geometry}. Additionally, the Anomaly flow of $(n-1,n-1)$-forms has been extensively studied by Phong--Picard--Zhang~\cite{phong2018anomaly,phong2018aflows}.

	This paper is devoted to the Chern--Ricci flow, which is an evolution equation of Hermitian metrics on a complex manifold by their Chern--Ricci form, first introduced by Gill~\cite{gill2011convergence} in the setting of manifolds with vanishing first Bott--Chern class.
	Let $(X,\omega_0)$ be a compact $n$-dimensional Hermitian manifold. The {\it Chern--Ricci flow} $\omega=\omega(t)$ starting at $\omega_0$ is an evolution equation of Hermitian metrics \begin{align}\label{crf}
	\frac{\partial \omega}{\partial t}=-{\rm Ric}(\omega), \quad\omega|_{t=0}=\omega_0,
	\end{align}
	where ${\rm Ric}(\omega)$ is the {\em Chern--Ricci form} of $\omega$ associated to the Hermitian metric $g=(g_{i\bar{j}})$, which in local coordinates is given by \[{\rm Ric}(\omega)=-\dc\log\det(g) .\] Here $d=\partial+\Bar{\partial}$ and $d^c=i(\Bar{\partial}-\partial)/2$ are both real operators, so that $\dc=i\partial\Bar{\partial}$. In the K\"ahler setting, ${\rm Ric}(\omega)=iR_{j\bar{k}}dz_j\wedge d\bar{z}_k$, where $R_{j\bar{k}}$ is the usual Ricci curvature of $\omega $. Thus, if $\omega_0$ is K\"ahler i.e., $d\omega_0=0$, \eqref{crf} coincides with the K\"ahler--Ricci flow. For complex manifolds with $c_1^{\rm BC}(X)=0$, Gill~\cite{gill2011convergence} proved the longtime existence of the flow and smooth convergence of the flow to the unique Chern--Ricci flat metric in the $\partial\bar{\partial}$-class of the initial metric.
	For general complex manifolds,	Tosatti and Weinkove \cite[Theorem 1.3]{tosatti2015evolution} characterized the maximal existence time $T_{\max}$ of the flow as \[T_{\max}:=\sup\{t>0: \exists\;\psi\in\mathcal{C}^\infty(X)\; \text{with}\; \omega_0-t\textrm{Ric}(\omega_0)+\dc\psi>0\}.\]
	
	\subsection*{Finite time singularities}
Suppose that the flow~\eqref{crf} exists on a maximal interval $[0, T_{\max})$ with $T_{\max}<\infty$, so the flow develops a finite time singularity. 
	We say that the Chern--Ricci flow does not develop a singularity at a point $x\in X$ if there exist an open neighborhood $U\ni x$ and a smooth metric $\omega_{T_{\max}}$ on $U$ such that $\omega(t)$ converges to $\omega_{T_{\max}}$ in $\mathcal{C}^\infty_{\rm loc}(U)$ as $t\to T^-_{\max}$. 
	
	The following question was asked by Feldman--Ilmanen--Knopf~\cite[Question 2, page 204]{feldman2003rotationally} for the K\"ahler--Ricci flow and by Tosatti--Weinkove~\cite[Question 6.1]{tosatti2022chern}
	for the Chern--Ricci flow.
	\begin{question}\label{ques: tosatti-weinkove}
		Do singularities of the Chern--Ricci flow form a union of all analytic subvarieties of $X$ for which the volume shrinks to zero as $t\to T_{\max}$? 
	\end{question}
	In the K\"ahler setting, this question was affirmatively answered by Collins--Tosatti~\cite{collins2015kahler}. When $X$ is a compact complex surface and $\omega_0$ is Gauduchon, i.e., $\dc\omega_0=0$, the Chern--Ricci flow preserves Gauduchon (pluriclosed) condition, in particular, the limiting form $\alpha_{T_{\max}}=\omega_0-T_{\max}\textrm{Ric}(\omega_0)$ is Gauduchon. The answer is thus affirmative in this case, due to Gill--Smith~\cite{gill2015behavior} (cf. also~\cite{tosatti2013chern}) where they proved that singularities of the Chern--Ricci flow form a finite union of disjoint (-1)-curves. 

	We partially answer Question~\ref{ques: tosatti-weinkove} under two additional assumptions.
	First, we assume that the limiting form $\alpha_{\tmax}$ is {\em uniformly non-collapsing}:\begin{equation}
	\int_X(\alpha_{T_{\max}}+\dc\psi)^n\geq c_0>0,\;\forall\,\psi\in  \mathcal{C}^{\infty}(X), \alpha_{T_{\max}}+\dc\psi>0.\end{equation} 
	We mention that when $\dim X=2$ and $\omega_0$ is a Gauduchon metric on $X$, the latter condition is equivalent to $\int_X\alpha_{\tmax}^2>0$ (by Stokes' theorem). A simple example (cf.~\cite[Remark 3.1]{tosatti2013chern}) where this condition appears is the following. Let $Y$ be a compact Hermitian manifold and $\pi: X\to Y$ be the blowup of a point with exceptional divisor $E$. Let $\omega_X$ and $\omega_Y$ be Gauduchon metrics on $X$ and $Y$ respectively, and fix $T_{\max }>0$. It is known that there is a metric $h$ on the line bundle $\mathcal{O}(E)$ with curvature $R_h$ such that for $C>0$ large enough, $\omega'= C\pi^*\omega_Y-T_{\max} R_h+\dc f$ is a Hermitian metric for some $f\in\mathcal{C}^\infty(X)$. By the adjunction formula, we can choose
	$$\omega_0:=(C+1)\pi^*\omega_Y+T_{\max}\textrm{Ric}(\omega_X)+\dc f$$ which is a Gauduchon metric. 
	Hence, $\alpha_{\tmax}=\pi^*\widetilde\omega_Y+\dc \widetilde{f}$ for some Gauduchon metric $\widetilde{\omega}_Y$ and $\Tilde{f}\in\mathcal{C}^\infty(X)$; cf. ~\cite[Lemma 3.2]{tosatti2013chern} or~\cite{buchdahl00NM}. For any $\psi\in\mathcal{C}^\infty(X)$, $$\int_X(\alpha_{T_{\max}}+\dc\psi)^2=\int_X\alpha_{\tmax}^2= \int_Y\widetilde\omega_Y^2>0.$$ 
	
	The second assumption is that $X$ has \emph{the bounded mass property}, that is, there exists a Hermitian metric $\omega_X$ such that $v_{+}(\omega_X)<+\infty$ (cf. Definition~\ref{def; mass}). This condition is automatically satisfied for compact complex surfaces (cf.~\cite{guedj2022quasi2}). For further examples of non-K\"ahler manifolds in higher dimensions, we refer the reader to \cite{angella2022plurisigned}. 
	Our main theorem is the following. 
	\begin{bigthmm}\label{thmA}
		Let $(X,\omega_0)$ be an $n$-dimensional compact Hermitian manifold with bounded mass property, i.e., $v_+(\omega_0)<+\infty$. 
		Assume that the Chern--Ricci flow~\eqref{crf} starting at $\omega_0$ exists on the maximal interval $[0,T_{\max})$ with $T_{\max}<\infty$, and that the limiting form $\alpha_{T_{\max}}$ is uniformly non-collapsing:
		\begin{equation}\label{eq: unifnoncollap}
		\int_X(\alpha_{T_{\max}}+\dc\psi)^n\geq c_0>0,\;\forall\,\psi\in\mathcal{C}^{\infty}(X)\,\text{such that}\,\alpha_{T_{\max}}+\dc\psi>0.\end{equation} 
		Then as $t \to T^{-}$ the metrics $\omega(t)$ converge to $\omega_{T_{\max}}$ in $\mathcal{C}^\infty_{\rm loc}(\Omega)$ for some Zariski open set $\Omega\subset X$.  
	\end{bigthmm}
	The strategy of the proof is as follows. Using the uniformly non-collapsing condition of $\alpha_{\tmax}$, we show that there exists a quasi-plurisubharmonic function $\rho$ with analytic singularities such that $\alpha_{\tmax}+\dc\rho$ dominates a hermitian metric. This form is called {\em big} (cf. Definition~\ref{def: big}).
	Then $\Omega$ is the set in which $\rho$ is smooth. In particular, it is Zariski open. Next, we establish several uniform local estimates for $\omega$ near the maximal time $\tmax$, adapting techniques from~\cite{collins2015kahler,gill2011convergence}. The convergence result follows directly from these estimates. 
	
	

	\subsection*{Degenerate parabolic complex Monge--Amp\`ere equations} In the previous paragraph, we studied the behavior of the Chern--Ricci flow at finite singularity time. It is natural to ask whether the flow can pass through this singularity.
	To do this, we need to define weak solutions of the Chern--Ricci flows starting from degenerate initial currents on a compact complex variety with mild singularities.
	This leads us to consider several geometric settings arising in the minimal model program, particularly the case of complex varieties with Kawamata log terminal (klt) singularities. From an analytic point of view, this situation naturally involves densities that may blow up but still belong to $L^p$ spaces for some exponent $p > 1$ whose size depends on the algebraic nature of the singularities.
	
	\smallskip
	On a compact Hermitian $n$-manifold $(X,\omega_X)$,
	we consider the following degenerate parabolic complex Monge--Ampère equation
	\begin{equation}\label{pcmae}
	\frac{\partial\f_t}{\partial t}=\log\left[ \frac{(\theta_t+\dc\f_t)^n}{\mu}\right], 
	\end{equation} for $t\in(0,\tmax)$, where $\tmax<\infty$ and
	\begin{itemize}
		\item $\theta_t=\theta+t\chi$ is an affine family of smooth semi-positive forms, where $\chi$ is smooth (1,1) form and $\theta$ is a smooth, {\em big} (1,1) form, that is there is a quasi-plurisubharmonic function $\rho$ with analytic singularities such that
		\[\theta+\dc \rho\geq \delta\omega_X\; \text{for some}\, \delta>0;\]
		\item $\mu$ is a positive measure on $X$ of the form 
		\[\mu=e^{{\psi^+-\psi^-}}dV_X\] with $\psi^{\pm}$ quasi-plurisubharmonic functions, being smooth on a given Zariski open subset $U\subset \{\rho>-\infty\}$ and $e^{-\psi^-}\in L^p$ for some $p>1$ and $dV_X$ a smooth volume form;
		\item $\f:[0,T_{\max}]\times X\rightarrow \mathbb{R}$ is the unknown function, with $\f_t:=\f(t,\cdot)$. 
	\end{itemize}

	We first define the weak solution of the Chern--Ricci flow:
	\begin{definition}
		A family of functions $\f_t:X\to\mathbb{R}$ for $t\in (0,\tmax)$ is said to be a weak solution of the  equation~\eqref{pcmae} starting with $\f_0$ if the following hold:
		\begin{enumerate}
			\item for each $t$, $\f_t$ is $\theta_t$-plurisubharmonic on $X$;
			\item $\f_t\to \f_0$ in $L^1(X)$ as $t\to 0^{+}$;
			\item for each $\varepsilon>0$ there exists a Zariski open set $\Omega_\varepsilon\subset X$ such that the function  $(t,x)\mapsto\f(t,x)\in \mathcal{C}^\infty([\varepsilon,\tmax-\varepsilon]\times\Omega_\varepsilon)$. Furthermore, the equation~\eqref{pcmae} satisfies in the classical sense on $[\varepsilon,\tmax)\times\Omega_\varepsilon$.
		\end{enumerate}
	\end{definition}
	
	The following theorem establishes the existence of the complex Monge--Amp\`ere flow starting with an initial function $\f_0$ with small Lelong numbers.
	\begin{bigthmm}\label{thmB} Let $(X,\omega_0)$ be an $n$-dimensional compact Hermitian manifold and $\theta$ a semi-positive and big $(1,1)$-form.
		Let $\f_0$ be an $\theta$-plurisubharmonic function satisfying $p^*/2c(\f_0)<\tmax$, where $p^*$ is the conjugate exponent of $p$. Then, there exists a weak solution $\f$ of the flow~\eqref{pcmae} starting with $\f_0$ for $t\in(0,\tmax)$.
	\end{bigthmm}
	Here, $c(\f_0)$ denotes the integrability index of $\f_0$, which is the supremum of positive constants $c>0$ such that $e^{-2c\f_0}$ is locally integrable. Thanks to Skoda's integrability theorem, $c(\f_0)=+\infty$ if and only if $\f_0$ has zero Lelong numbers at all points.
	
	\medskip
	
	Let us briefly outline the strategy for the proof of Theorem~\ref{thmB}. We first approximate $\f_0$ by a decreasing sequence of smooth $(\theta+2^{-j}\omega_X)$-plurisubharmonic functions $\f_{0,j}$ thanks to Demailly's regularization theorem. Similarly, $\psi^{\pm}$ are approximated by smooth quasi-plurisubharmonic functions. We consider the corresponding solution $\f_{t,j}$ to the equation~\eqref{pcmae}, with $\theta_{t,j}=\theta_t+2^{-j}\omega_X$.  Our goal is to establish several a priori estimates that allow us to take the limit as $j\to+\infty$. Precisely, we aim to show that for any $\varepsilon>0$, there is a Zariski open set $\Omega_\varepsilon\subset X$ such that for each fixed $0<T<\tmax$ and any compact subset $K\subset\Omega_\varepsilon$,
	\begin{itemize}
		\item $\|\f_{t,j}\|_{\mathcal{C}^0([\varepsilon,T]\times K)}\leq C_{\varepsilon,T,K}$;
		\item $\partial_t\f_{t,j}$ is uniformly bounded on $[\varepsilon,T]\times K$;
		\item $\Delta_{\omega_X}\f_{t,j}$ is uniformly bounded on $[\varepsilon,T]\times K$.
	\end{itemize}
	We then apply the parabolic Evans--Krylov--Trudinger theory and Schauder estimates to obtain uniform higher-order local estimates for all derivatives of $\f_{t,j}$ (cf.~\cite{gill2011convergence} for a recent account in the Chern--Ricci flow context). This allows us to pass to the limit and conclude that $$\f_{t,j}\to \f_t\in\mathcal{C}^\infty([\varepsilon, T]\times \Omega_\varepsilon)$$ as $j\to +\infty$. Furthermore, we automatically have the weak convergence $\f_t\to \f_0$ as $t\to 0^+$.  Stronger convergence results are discussed in Section~\ref{sect: conv_zero} when $\f_0$ has less singularity.
	
	\medskip
	We emphasize here that the mild assumption $p^*/2c(\f_0)<\tmax$ guarantees that the approximating flow is well-defined (i.e., not identically $-\infty$) and is crucial for the smoothing properties of the flow. As noted by Di Nezza and Lu in~\cite{di2017uniqueness} for the K\"ahler setting, without this assumption, the K\"ahler--Ricci flow may still run, but it is likely to lose its regularizing effect due to the presence of positive Lelong numbers. In such cases, they highlighted that the main challenge lies in establishing the a priori $\mathcal{C}^0$-estimate. Their approach relies on Kołodziej's method, which uses generalized Monge--Ampère capacities. In contrast, our approach follows the recent developments of Guedj and Lu~\cite{guedj2021quasi1,guedj2021quasi}, which have the advantage of being applicable to degenerate (1,1) forms in the non-K{\"a}hler context.

	\medskip
	We finally apply the previous analysis to treat the case of mildly singular varieties. This allows us to define a good notion of the weak Chern--Ricci flow on complex compact varieties with log terminal singularities. We will discuss it in Section~\ref{sect: crf_lt} and prove the following. 
	\begin{bigthmm}\label{thmC}
		Let $Y$ be a compact complex variety with log terminal singularities. Assume that $\theta_0$ is a Hermitian metric such that 
		\begin{equation*}
		T_{\rm max}:=\sup \{t>0:\, \exists\; \psi\in\mathcal{C}^\infty(Y) \,\text{such that}\; \theta_0-t\textrm{Ric}(\theta_0)+\dc\psi >0 \}>0.
		\end{equation*} Assume that $S_0=\theta_0+\dc\phi_0$ is a positive $(1,1)$-current with sufficiently small slopes. Then, there exists a family $(\omega_t)_{t\in[0,\tmax)}$ of positive (1,1) current  on $Y$ starting with $S_0$ such that
		\begin{enumerate}
			\item $\omega_t=\theta_0-t\textrm{Ric}(\theta_0)+\dc\f_t$ are positive (1,1) currents;
			\item $\omega_t\to S_0$ weakly as $t\to 0^+$;
			\item for each $\varepsilon>0$ there exists a Zariski open set $\Omega_\varepsilon$ such that on $[\varepsilon,\tmax)\times \Omega_\varepsilon$,  $\omega$ is smooth and
			\[\frac{\partial\omega}{\partial t}=-\textrm{Ric}(\omega).\]
		\end{enumerate}
	\end{bigthmm}
	This generalizes previous results of Song--Tian~\cite{song2017kahler}, Guedj--Zeriahi~\cite{guedj2017degenerate}, T\^o~\cite{to2017regularizing}, DiNezza--Lu~\cite{di2017uniqueness}, Guedj--Lu--Zeriahi~\cite{guedj2020pluripotential} and the author \cite{dang2021pluripotential} to the non-K\"ahler case, and of \cite{to2018regularizing,nie2017weak} and the author~\cite{dang2021chern} to more degenerate initial data.

	\subsection*{Organization of the paper} We establish  a priori estimates in Section~\ref{sect: estimate}, which will be used to prove Theorem~\ref{thmB} in Section~\ref{sect: existence}. 
	While Theorem~\ref{thmA} will be proved in Section~\ref{sect: finite_sing}, studying the behavior of the Chern--Ricci flow at non-collapsing finite time singularities.
	In Section~\ref{sect: crf_lt}, we apply these tools to prove the existence of the weak Chern--Ricci flow with initial degenerate data on compact complex varieties with log terminal singularities, proving Theorem~\ref{thmC}.
	
	\subsection*{Acknowledgement} 
	The author is grateful to Chung-Ming Pan for carefully reading the first draft. Special thanks go to Jacopo Stoppa and T\^at-Dat T\^o for valuable comments. The author would also like to thank the anonymous referees for their comments and suggestions.
	This work is partially supported by the project PARAPLUI ANR-20-CE40-0019.
	\section{Preliminaries}

	\subsection{Recap on pluripotential theory}
	Let $X$ be a compact complex manifold of dimension $n$, equipped with a Hermitian metric $\omega_X$. We fix $\theta$ a smooth semi-positive real $(1,1)$-form on $X$.
	\subsubsection{Quasi-plurisubharmonic functions and Lelong numbers}
	A function $u\in L^1(X)$ is quasi-plurisubharmonic (quasi-psh for short) if it is locally given as the sum of a smooth function and a plurisubharmonic (psh for short) function.
	\begin{definition}
		A quasi-psh function $\f:X\to[-\infty,+\infty)$ is called {\em $\theta$-plurisubharmonic} ({\em $\theta$-psh} for short) if it satisfies $\theta_\f:=\theta+\dc\f\geq 0$ in the weak sense of currents.
		We let $\PSH(X,\theta)$ denote the set of all $\theta$-psh functions that are not identically $-\infty$.
	\end{definition} The set $\PSH(X,\theta)$ is endowed with the $L^1(X)$-topology. By Hartogs' lemma, the map $\f \mapsto \sup_X \f$ is continuous with respect to this topology. Since the set of closed positive currents in a fixed  $\dc$-class is compact (in the weak topology), it follows that the set of $\f\in\PSH(X,\theta)$, with $\sup_X \f=0$ is compact. 
	We refer the reader to~\cite{demaillycomplex,guedj2017degenerate} for basic properties of $\theta$-psh functions. 
	\smallskip
	
	Quasi-psh functions are, in general, singular, and a convenient way to measure their singularities is the Lelong numbers.
	\begin{definition}
		Let $x_0\in X$. Fixing a holomorphic chart $x_0\in V_{x_0}\subset X$,  the {\em Lelong number} $\nu(\f,x_0)$ of a quasi-psh function $\f$ at  $x_0\in X$ is defined as follows:
		\begin{align*}
		\nu(\f,x_0):=\sup\{\gamma\geq 0: \f(z)\leq \gamma\log\|z-x_0\|+O(1), \; \text{on}\; V_{x_0}\}.  
		\end{align*}
	\end{definition}
	We remark here that this definition does not depend on the choice of local holomorphic charts.
	In particular, if $\f=\log|f|$ in a neighborhood $V_{x_0}$ of $x_0$, for some holomorphic function $f$, then $\nu(\f,x_0)$ is equal to the vanishing order $\textrm{ord}_{x_0}(f):=\sup\{ k\in\mathbb{N}:D^\gamma f(x_0)=0,\forall\, |\gamma|<k \}$.
	\smallskip
	
	In some contexts, it is more convenient to work with the integrability index rather than the Lelong numbers. The \emph{integrability index} of a quasi-psh function $\f$ at a point $x\in X$ is defined by \begin{equation*}
	c(\f,x):=\sup\{c>0: e^{-2c\f}\in L^1(V_x) \}
	\end{equation*} where $V_x$ is some neighborhood around $x$. 
	This definition does not depend on the choice of the open neighborhood $V_x$.
	We denote by $c(\f)$ the infimum of $c(\f,x)$ for all $x\in X$. Since $X$ is compact, it follows that $c(\f)>0$.
	
	Skoda's integrability theorem (cf.~\cite[Chapter 2]{guedj2017degenerate}) yields the following "optimal"  relation between the Lelong number of a quasi-psh function $\f$ at a point $x_0\in X$ and the local integrability  index of $\f$ at $x_0$: \begin{equation}
	\frac{1}{\nu(\f,x_0)}\leq c(\f,x_0)\leq\frac{n}{\nu(\f,x_0)}.
	\end{equation} 
	In  particular, $c(\f)=+\infty$ if  and only if $\nu(\f,x)=0$ for all $x\in X$.
	
	\subsubsection{Monge--Amp\`ere measures}
	
	The complex Monge--Amp\`ere measure $(\theta+\dc u)^n$ is well-defined for any $\theta$-psh function $u$ which is bounded, as follows from the Bedford--Taylor theory: if $\beta=\dc\rho$ is a K\"ahler form such that $\beta>\theta$  in a local open chart $U\subset X$, then $u$ is $\beta$-psh and the positive currents $(\beta+\dc u)^j$ are well-defined for $1\leq j\leq n$. Thus, the {\em complex Monge--Amp\`ere measure}
	\begin{equation*}
	(\theta+\dc u)^n:=\sum_{j=0}^n \binom{n}{j}(\beta+\dc u)^j\wedge (\theta-\beta)^{n-j}
	\end{equation*} is a positive measure on $X$.
	Indeed, by Demailly's regularization theorem, we can approximate $u$ by a decreasing sequence of smooth $(\theta+\varepsilon_j\omega_X)$-psh functions $u_j$. Consequently, $(\theta+\dc u)^n$ is the limit of positive measures $(\theta+\varepsilon_j\omega_X+\dc u_j)^n$, ensuring that $(\theta+\dc u)^n$ is positive.
	
	This definition does not depend on the choice of $\beta$ by the same arguments.
	We refer to~\cite{dinew2012pluripotential} for an adaptation of~\cite{bedford1976dirichlet, bedford1982new} to the Hermitian context.
	We recall the following maximum principle.
	\begin{lemma}\label{lem: max-princ}
		Let $\f,\p$ be bounded $\theta$-psh functions such that $\f\leq \p$. Then
		\begin{equation*}
		\mathbf{1}_{\{\f=\p\}}(\theta+\dc\f)^n\leq\mathbf{1}_{\{\f=\p\}}(\theta+\dc\p)^n.
		\end{equation*}
	\end{lemma}
	\begin{proof} This is a direct consequence of Bedford--Taylor's maximum principle; see~\cite[Theorem 3.23]{guedj2017degenerate}.
		We refer the reader to~\cite[Lemma 1.2]{guedj2022quasi2} for a brief proof.
	\end{proof}
	\subsubsection{Positivity assumptions}
	
	For our purposes, we need to assume a slightly stronger positivity property of the form $\theta$ in the sense of~\cite{guedj2021quasi}. 
	\begin{definition}\label{def; mass}
		We consider
		\[v_{-}(\theta):=\inf\left\{\int_X(\theta+\dc\f)^n: \f\in\PSH(X,\theta)\cap L^\infty(X) \right\}\]
		and \[v_{+}(\theta):=\sup\left\{\int_X(\theta+\dc\f)^n: \f\in\PSH(X,\theta)\cap L^\infty(X) \right\}.\]
	\end{definition}
We emphasize that when $\theta$ is Hermitian, the supremum and infimum in the definition of these quantities can be taken over $\PSH(X,\theta)\cap \mathcal{C}^\infty(X)$ due to Demailly's regularization theorem and Bedford--Taylor's convergence results. 
	\begin{definition}
		A function $\rho$ is said to have {\em analytic singularities} if there exists a constant $c>0$ such that locally on $X$, $$\rho=c\log\sum_{j=1}^N|f_j|^2+O(1)$$ where the $f_j$'s are holomorphic functions.
	\end{definition}
	\begin{definition}\label{def: big}
		We say $\theta$ is {\em big}  if there exists a $\theta$-psh function with analytic singularities such that $\theta+\dc \rho\geq \delta\omega_X$ for some $\delta>0$. We let $\Omega$ denote the open Zariski set in which $\rho$ is locally bounded.
	\end{definition}	
	Such a form appears in some contexts of complex differential geometry. For example, if $Y$ is a compact complex space endowed with a Hermitian metric $\omega_Y$ and $\pi: X\to Y$ is a log resolution of singularities, then the form $\theta:=\pi^*\omega_Y$ is big; see, e.g.,~\cite[Proposition 3.2]{fino2009blow}. Moreover, we can find  a $\theta$-psh function $\rho$ with analytic singularities such that $\theta+\dc\rho\geq \delta\omega_X$, and
	\begin{equation*}
	\Omega=\{\rho>-\infty\}=X\setminus \textrm{Exc}(\pi)=\pi^{-1}(Y_{\rm reg})\simeq Y_{\rm reg}.
	\end{equation*}
	
	\subsubsection{Envelopes}
	Recall that a  Borel set $E\subset X$ is {\em (locally) puripolar} if for each $x\in X$, there exists an open neighborhood $U$ of $x$ and a psh function $u$ on $U$ such that $E\cap U\subset \{u=-\infty\}$. 
	As follows from \cite[Theorem 1.1]{vu2019pluripolar} or \cite[Lemma 2.6]{guedj2022quasi2}, the set $E$ is globally pluripolar, i.e., there exists  $u\in \PSH(X,\omega_X)$ such that $E\subset\{u=-\infty \}$. Since $\theta$ is big, the function $u':=\delta u+\rho$ is $\theta$-psh and its $-\infty$-locus contains $E$.

	\begin{definition}
		Given a measurable function $h:X\to\mathbb{R}$, we define the {\em $\theta$-psh envelope} of $h$ by
		\begin{equation*}
		P_\theta(h):=(\sup\{u\in\PSH(X,\theta): u\leq h\;\text{on}\, X \})^*
		\end{equation*} where the star means that we take the upper semi-continuous regularization.
	\end{definition}
We note that this definition is equivalent to the one given in \cite[Definition 2.2]{guedj2022quasi2}; cf. \cite[Corollary 2.7]{guedj2022quasi2}.

	We have the following result, established in~\cite[Theorem 2.3]{guedj2022quasi2}.
	
	\begin{theorem}\label{thm: envelope}
		If $h$ is bounded from below, quasi-lower-semi-continuous, and  $P_\theta(h)<+\infty$, then
		\begin{enumerate}
			\item $P_\theta(h)$ is a bounded $\theta$-psh function;
			\item $P_\theta(h)\leq h$ in $X\setminus P$, for some pluripolar set $P$;
			\item $(\theta+\dc P_\theta(h))^n$ is concentrated on the contact set $\{P_\theta(h)=h\}$.
		\end{enumerate}
	\end{theorem}
	
	The following $\mathcal{C}^0$-estimate is crucial in the sequel.
	\begin{lemma}\label{lem: C0_estimate} Let $\theta$ be a smooth real semi-positive and big (1,1)-form.
		Assume $\f\in\PSH(X,\theta)\cap L^\infty(X)$ satisfies \[(\theta+\dc\f)^n\leq e^{A\f-g}fdV_X,\] where $A>0$ and $f$, $g$ are measurable functions such that $e^{A\psi-g}f\in L^q(X)$ with $q>1$, for some $\psi\in\PSH(X,\delta\theta)$, with $\delta\in (0,1)$. Then we have the following estimate 
		\begin{equation*}
		\f\geq \psi-C
		\end{equation*} where $C$ is a positive constant only depending on $n$, $A$, $\delta$, $\theta$, $q$ and an upper bound for $\int_X e^{q(A\psi-g)}f^qdV_X$.
	\end{lemma}
	\begin{proof}We apply the approach recently developed by Guedj--Lu~\cite{guedj2021quasi1,guedj2021quasi}. 
		Subtracting a large constant, we can assume that $\f\leq 0$.
		Set $u:=P_{(1-\delta)\theta}(\f-\psi)$. Fix $M>0$ so large that $E:=\{\psi >-M\}$ is not empty, and hence it is non-pluripolar. 
		We claim that the global extremal function $V_{E,(1-\delta)\theta}^*$ of $E$ is not identically $+\infty$, where
$$V_{E,(1-\delta)\theta}(x):=\sup\{\varphi(x): \varphi\in\PSH(X,(1-\delta)\theta), \varphi\leq 0\,\text{on}\, E \}.$$ 
The proof follows almost verbatim from \cite[Theorem 9.17]{guedj2017degenerate}. We suppose by contradiction that $\sup_X V_{E,(1-\delta)\theta}=+\infty$. By  a lemma of Choquet (see \cite[Lemma 4.31]{guedj2017degenerate}), there exist an increasing sequence $u_j\in \PSH(X,(1-\delta)\theta)$ such that  $u_j=0$ on $E$, $\sup_X u_j\geq 2^j$, and $$V_{E,(1-\delta)\theta}=(\lim\nearrow u_j)^*.$$ Set $v_j:=u_j-\sup_X u_j$. These functions belong to the compact set of $(1-\delta)\theta$-psh functions normalized by $\sup_X w=0$. Hence, there exists a uniform constant $C>0$ such that $\int_X v_jdV\geq -C$; cf. \cite[Proposition 2.1]{dinew2012pluripotential}. Since $(1-\delta)\theta\geq 0$, the function $v:=\sum_{j\geq 1}2^{-j}v_j\in \PSH(X,(1-\delta)\theta)$ is a decreasing limit of functions in $\PSH(X,(1-\delta)\theta)$, with $\int_X vdV\geq -C$. Since $v=-\infty$ on $E$, it follows that $E$ is $\PSH(X,(1-\delta)\theta)$-pluripolar. This gives a contradiction.

		Since $u\leq \f-\psi\leq M$ on $E$, hence $u-M$ is a candidate defining $V_{E,(1-\delta)\theta}$.
	Therefore, $\sup_X u\leq M+\sup_X V_{E,(1-\delta)\theta}^*$ is uniformly bounded from above. 
		
		Since $\f-\psi$ is bounded from below and quasi-continuous, it follows from Theorem~\ref{thm: envelope} that $((1-\delta)\theta+\dc u)^n$ is supported on the contact set $D:=\{u+\p=\f \}$.
		We observe that $u+\p$ and $\f$ are both $\theta$-psh functions satisfying $u+\p\leq \f$, it follows from Lemma~\ref{lem: max-princ} that
		\begin{equation*}
		\mathbf{1}_D(\theta+\dc (u+\p))^n\leq \mathbf{1}_D(\theta+\dc\f)^n.
		\end{equation*}
		From these, we have
		\begin{equation*}
		\begin{split}
		((1-\delta)\theta+\dc u)^n&=\mathbf{1}_D((1-\delta)\theta+\dc u)^n\\
		&\leq \mathbf{1}_D(\theta+\dc (u+\p))^n\\
		&\leq\mathbf{1}_D(\theta+\dc\f)^n\\
		&\leq  \mathbf{1}_D e^{A\f-g}fdV_X\\
		&= \mathbf{1}_D e^{A u}e^{A\p-g}fdV_X.
		\end{split}
		\end{equation*} By assumption, $F:=e^{A\p-g}f\in L^q(X)$, with $q>1$. 
		Since $(1-\delta)\theta$ is semi-positive and big, it follows from~\cite[Lemma 2.1]{guedj2021quasi} that there exists a uniform constant $m>0$ only depending on $dV_X$, $n$, $q$, $\theta$, $\delta$, and $\|e^{A\p-g}f\|_{L^q}$, such that we can find $v\in\PSH(X,(1-\delta)\theta)\cap L^\infty(X)$ satisfying $-1\leq v\leq 0$ and $$((1-\delta)\theta+\dc v)^n\geq mF dV_X.$$ Hence
		$$e^{-A(v+A^{-1}\ln m)}((1-\delta)\theta+\dc v)^n\geq FdV_X\geq e^{-Au}((1-\delta)\theta+\dc u)^n.$$
		The domination principle (cf. \cite[Proposition 1.14]{guedj2021quasi}) yields $u\geq v+A^{-1}\ln m$.
		This completes the proof.
	\end{proof}
	\subsection{Equisingular approximation}
	Fix $\f$ a $\theta$-psh function on $X$. We aim at approximating $\f$ by a decreasing sequence of quasi-psh functions which are less singular than $\f$ and such that their singularities are somehow comparable to those of $\f$. This leads us to apply Demailly's equisingular approximation theorem. 
	For each $c>0$, we  define the \emph{Lelong super-level sets}  \[E_c(\f):= \{x\in X:\nu(\f,x)\geq c \}.\]  
	We also use the notation $E_c(T)$ for a closed positive $(1,1)$-current $T$.
	A well-known result of Siu~\cite{siu1974analyticity}  asserts that the Lelong super-level sets
	$E_c(\f)$ are    analytic subsets of $X$.  We refer the reader to~\cite[Remark 3.2]{demailly1992regularization} for an alternative proof.
	
	The following result of Demailly on the equisingular approximation of a quasi-psh function by quasi-psh functions with analytic singularities is crucial.
	
	\begin{theorem}[Demailly's equisingular approximation]\label{thm: dem} 
		Let $\f$ be a $\theta$-psh function on $X$. There exists a decreasing sequence of quasi-psh functions $(\f_m)_{m\in\mathbb{N}}$ such that
		\begin{enumerate}
			\item $(\f_m)$ converges pointwise and in $L^1(X)$ to $\f$ as $m\to+\infty$, 
			\item $\f_m$ has the same
			singularities as  $1/2m$ times a logarithm of a sum of squares of holomorphic functions,
			\item $\dc\f_m\geq -\theta -\varepsilon_m\omega_X$, where $\varepsilon_m>0$ decreases to 0 as $m\to+\infty$,
			\item $\int_Xe^{2m(\f_m-\f)}dV<+\infty$,
			\item $\f_m$ is smooth outside the analytic subset $E_{1/m}(\f)$.
		\end{enumerate}
	\end{theorem}
	\begin{proof}
		We briefly outline the idea for the reader's convenience, as it is likely already known to experts. We follow the proof of ~\cite{demailly1992regularization} by applying with the current $T=\dc\f$ and the smooth real (1,1) form $\gamma=-\theta$. We also borrow notation from there.
		
		For
		$\delta>0$ small, let us cover $X$ by $N=N(\delta)$ geodesic balls $B_{2r}(a_j)$ with respect to $\omega_X$ such that $X=\cup_jB_r(a_j)$ and in terms of coordinates $z^j=(z_1^j,\ldots, z_n^j)$,
		\[\sum_{\ell=1}^n\lambda_{\ell}^j i dz_\ell^j\wedge d\Bar{z}_\ell^j \leq \gamma|_{B_{2r}(a_j)}\leq \sum_{\ell=1}^n(\lambda_{\ell}^j+\delta)i dz_\ell^j\wedge d\Bar{z}_\ell^j\]
		where we have  diagonalized $\gamma(a_j)$ at
		the center $a_j$. Here, $N$ and $r$ are taken to be uniform. Set $\f^j:=\f|_{B_{2r}(a_j)}-\sum_{\ell=1}^n\lambda_\ell^j|z^j_\ell|^2$.
		On each $B_{2r}(a_j)$, we define
		\[\f_{j,\delta, m}:=\frac{1}{2m}\log \sum_{k\in \mathbb{N}}|f_{j,m,k}|^2, \] where $(f_{j,m,k})_{k\in\mathbb{N}}$ is an orthogonal basis of the Hilbert space $\mathcal{H}_{B_{2r}(a_j)}\left(m\f^j\right)$ of holomorphic functions on $B_{2r}(a_j)$ with finite $L^2$ norm $\|u\|^2=\int_{B_{2r}(a_j)}|u|^2e^{-2m\f^j}dV(z^j)$. Note that since $\dc\f\geq \gamma$ it follows that $\f-\sum_{\ell=1}^n\lambda_\ell^j|z^j_\ell|^2$ is psh on $B_{2r}(a_j)$. The Bergman kernel process applied on each ball $B_{2r}(a_j)$ has provided approximations $\f_{j,\delta,m}$ of $\f^j=\f|_{B_{2r}(a_j)}-\sum_{\ell=1}^n\lambda_\ell^j|z^j_\ell|^2$, it thus remains to glue these functions  into a function $\f_{\delta,m}$ globally defined on $X$. For this, we set
		\[ \f_{\delta,m}(x)= \frac{1}{2m}\log \left( \sum_j \theta_j(x)^2\exp\left(2m\left(\f_{j,\delta,m}+\sum_\ell(\lambda^j_\ell-\delta)|z_\ell^j|^2\right) \right) \right)\]
		where $(\theta_j)_{1\leq j\leq N}$ is the partition of unity subordinate to the $B_{r}(a_j)$'s. Now we take $\delta=\delta_m\searrow 0$ slowly and $\f_m=\f_{\delta_m,m}$ the same computations as in~\cite[page 16]{demailly1992regularization} ensure that $$\dc\f_m\geq \gamma-\varepsilon(\delta_m)\omega_X$$ for $m\geq m_0$ sufficiently large and $\varepsilon_m=\varepsilon(\delta_m)\searrow 0$ as $m\to+\infty$. By construction, the properties (1), (2), (3), and (5) are satisfied.

		The property (4) is crucial for later use, its proof should be provided. 
		The argument originates from~\cite[Theorem 2.3, Step 2]{demailly2001psef}, using local uniform convergence and the strong
		Noetherian property. By the properties of the functions $\f_m$, it suffices to show that on each ball $B_j=B_r(a_j)$,
		\[\int_{B_j}e^{2m\f_m-2m\f}dV=\int_{B_j}\left(\sum_{k\in \mathbb{N}}|f_{j,m,k}|^2\right)e^{-2m\f}dV(z^j)<+\infty.\]
		We let $\mathcal{F}_1\subset \mathcal{F}_2\subset\ldots\mathcal{F}_k\subset \ldots \subset \mathcal{O}(B_{2r}(a_j)\times B_{2r}(a_j))$ denote the sequence of ideal coherent sheaves generated by the holomorphic functions $\left(f_{j,m,\ell} (z)\overline{f_{j,m,\ell}(\bar{w})}\right)_{\ell\leq k}$ on $B_{2r}(a_j)\times B_{2r}(a_j)$. By the strong Noetherian property
		(see e.g.,~\cite[C. II, 3.22]{demaillycomplex}) 
		the sequence $(\mathcal{F}_k)$ 
		is stationary on a compact subset $B_j\times B_j\subset\subset B_{2r}(a_j)\times B_{2r}(a_j)$ at an index $k_0$ large enough. Using the Cauchy--Schwarz inequality we have that the sum of the series $U(z,w)=\sum_{k\in\mathbb{N}} f_{j,m,k} (z)\overline{f_{j,m,k}(\bar{w})}$ is bounded from above by 
		\[ \left(\sum_{k\in \mathbb{N}}|f_{j,m,k}(z)|^2 \sum_{k\in \mathbb{N}}|f_{j,m,k}(\Bar{w})|^2\right)^{\frac{1}{2}}\]
		hence uniformly convergent on every compact subset of $B_{2r}(a_j)\times B_{2r}(a_j)$. Since the space of sections of a coherent ideal sheaf is closed under the topology of uniform convergence on compact subsets, the Noetherian property guarantees $U(z,w)\in\mathcal{F}_{k_0}(B_j\times B_j)$.  Restricting to the conjugate diagonal $w = \Bar{z}$, we obtain \[ \sum_{k\in \mathbb{N}}|f_{j,m,k}(z)|^2\leq C_0\left(\sum_{k\leq k_0}|f_{j,m,k}(z)|^2\right)\]
		on $B_j$.  Since all terms $f_{j,m,k}$ have the $L^2$-norm equal to 1 with respect to the weight $e^{-2m\f}$, this completes the proof.
	\end{proof}
	
	Using this, one obtains the following lemma, which is slightly more general than the one in~{\cite{di2017uniqueness}}.
	
	\begin{lemma}\label{lem: equisingular} Let $\theta$ be a big (1,1) form. Assume  $\f\in\PSH(X,\theta)$. Then for each $\varepsilon>0$ there exist $c(\varepsilon)>0$ and $\psi_\varepsilon\in\PSH(X,\theta)\cap\mathcal{C}^\infty\left(X\setminus  (\{\rho=-\infty\}\cup E_{c(\varepsilon)}(\f)) \right)$ such that\begin{equation}\label{eq: finitemass}
		\int_X e^{\frac{2}{\varepsilon}(\psi_\varepsilon-\f)}dV_X<+\infty.
		\end{equation} 
	\end{lemma}
	
	\begin{proof} The proof is quite close to that of~\cite[Lemma 2.7]{di2017uniqueness}.
		Recall that the bigness of $\theta$ implies that there exists $\rho$ an $\theta$-psh function with analytic singularities and $\sup_X\rho=0$ such that \[\theta+\dc\rho\geq 3\delta_0\omega_X\quad \text {for a fixed constant}\; \delta_0>0.\] 
		Let $c(\f)$ be the integrability index of $\f$. We can assume that $c(\f)<+\infty$; otherwise we are done. By Theorem~\ref{thm: dem}, we can find $(\f_m)$ a Demailly's equisingular approximant of $\f$. We have that $\f_m$ is smooth in the complement of the analytic subset $E_{1/m}(\f)$ and \[\theta+\dc\f_m\geq-\varepsilon_m\delta_0\omega_X\]
		for $\varepsilon_m>0$ decreasing to zero as $m$ goes to $+\infty$. We notice that the errors $\varepsilon_m>0$ appear in the gluing process; see Theorem~\ref{thm: dem}. We choose $m=m(\varepsilon)$  to be the smallest positive integer such that
		\[m>\frac{2}{\varepsilon(1+\varepsilon_m)}, \quad\frac{2\varepsilon_m}{\varepsilon(1+\varepsilon_m)}<c(\f).\]
		We now set \begin{equation}\label{func}
		\psi_\varepsilon:=\frac{\f_m}{1+\varepsilon_m}+\frac{\varepsilon_m}{1+\varepsilon_m}\rho.
		\end{equation}
		Thus, we have $$\theta+\dc\psi_\varepsilon\geq \frac{\varepsilon_m}{1+\varepsilon_m}2\delta_0\omega_X:=2\kappa\omega_X.$$
		Holder's inequality ensures that~\eqref{eq: finitemass} holds, noticing that $\rho\leq 0$. We easily see that $\psi_\varepsilon$ is smooth in the complement of $\{\rho=-\infty\}\cup E_{c(\varepsilon)}(\f)$ with $c(\varepsilon)=m(\varepsilon)^{-1}$.
	\end{proof}

	\section{A priori estimates}\label{sect: estimate}

	\subsection{Notation}\label{sect: notation}
	We use some notation from~\cite[Section 3.1]{di2017uniqueness}. 
	Until further notice, $X$ denotes a compact complex manifold of dimension $n$, endowed with a reference Hermitian form $\omega_X$.
	Following the strategy outlined in the introductory section, we assume, in this part, that $\theta$ and $\theta_t=\theta+t\chi $, ${t\in(0,\tmax)}$, are Hermitian metrics, and $\f_0$ is a smooth strictly $\theta$-psh function. 
    We denote by $\mu:=fdV_X$  a positive measure with density $\|f\|_{L^p}\leq C$ uniformly, for some $p>1$. For higher-order estimates, we assume moreover that
	$$f=e^{\psi^+-\psi^-}$$
	where $\psi^\pm$ are smooth quasi-psh functions. Recall that $\rho$ is a $\theta$-psh function with analytic singularities such that $\theta+\dc\rho$ dominates a Hermitian form. We may assume that $\sup_X\rho=0$. We remark that as follows from~\cite{guedj2021quasi},  a priori bounds below remain valid when $\theta$ is semi-positive and big, $f\in L^p(X,\omega_X)$, and $\varphi_0$ is merely  bounded and $\theta$-psh.
	
	We consider $\f_t$  a smooth solution of the following parabolic complex Monge--Amp\`ere equation \begin{equation}\label{pcmae2}
	\frac{\partial\f_t}{\partial t}=\log\left[ \frac{(\theta_t+\dc\f_t)^n}{\mu}\right],\; \f|_{t=0}=\f_0
	\end{equation} on $[0,\tmax)$; see e.g.,~\cite{tosatti2015evolution}. We should keep in mind that $\f_t$ plays the role of its approximants $\f_{t,j}$ in establishing a priori estimates. For brevity, we will suppress the index $j$.
	

	We fix $T$ and $S$
	such that 
	\begin{equation*}
	\frac{p^*}{2c(\f_0)}<T<S<T_{\max }.
	\end{equation*} where $p^*$ is the conjugate exponent of $p$, i.e. $\frac{1}{p}+\frac{1}{p^*}=1$.
	Since we are interested in the behavior of the flow~\eqref{pcmae2} near zero, we can assume that $$\theta_S\geq (1-a)\theta,
	\quad \text{for}\; a\in [0,1/2).$$ It is truly natural in several geometric contexts; for example, $\theta_t$ are the pullback of Hermitian forms.
	Thus, for each $t\in[0,S]$, we have
	\begin{equation*}
	\theta_t=\frac{t\theta_S}{S}+\frac{S-t}{S}\theta\geq\left(1-\frac{at}{S}\right)\theta.
	\end{equation*}
	During the proof, we use the notation $\omega_t:=\theta_t+\dc\f_t$ for the smooth path of Hermitian forms and denote $\Delta_t=\tr_{\omega_t}\dc$ the corresponding time-dependent Laplacian operator on functions.
	
	We fix $\varepsilon_0>0$ small and let $\psi_0:=\psi_{\varepsilon_0}$ as established in Lemma~\ref{lem: equisingular}. By construction, $\psi_0$ is smooth outside an analytic subset $\{\rho=-\infty\}\cup E_{c(\varepsilon)}(\f_0)$ and satisfies
	\begin{equation}\label{eq: kappa}\theta+\dc\psi_0\geq 2\kappa\omega_X.\end{equation}
	We let $E_1$, $E_2$ denote the following quantities
	\begin{equation*}
	E_1:=\int_X e^{\frac{2(\psi_0-\f_0)}{\varepsilon_0}}dV_X<+\infty,
	\end{equation*}
	\begin{equation*}
	E_2:=\int_X e^{-\frac{p^*\psi_0}{T}}dV_X<+\infty.
	\end{equation*}
	Observe that $E_1$ is finite by Lemma~\ref{lem: equisingular}, while $E_2$ is finite since $\frac{p^*}{2c(\f_0)}<T$ and that $\psi_0$ is less singular than $\f_0$. We should emphasize that $\f_0$ in this a priori estimate section plays a role in its approximating sequence $\f_{0,j}$ (which are smooth strictly $\theta$-psh functions decreasing to $\f_0$). The corresponding sequence $E_1^j$ is uniformly bounded from above in $j$. Hence, we can pass the limit.
	
	In what follows, we use $C$ to denote a positive constant whose value may change from line to line but is uniformly controlled. 
	\subsection{Uniform estimate}
	
	We first look for an upper a priori bound for $\f_t$. We recall that \[\frac{1}{2}\theta\leq \theta_t\leq A\omega_X,\;\forall\, t\in [0,T],\] for $A>0$ sufficiently large.
	It follows from~\cite{tosatti2010complex} that there exist a constant $c$ and a smooth $A\omega_X$-psh function $\Phi$ normalized by $\inf_X\Phi=0$ such that 
	\[\left(A\omega_X+\dc\Phi\right)^n=e^cfdV_X.\]
	\begin{proposition}\label{prop: C0upper}
		For any $(t,x)\in[0,T]\times X$, there exists a uniform constant $C>0$ such that 
		\begin{equation*}
		\f_t(x)\leq C.
		\end{equation*}
	\end{proposition}
	\begin{proof}
		For any $(t,x)\in [0,T]\times X$, we set $v(t,x)=\Phi(x)+ct+\sup_X\f_0$. Then, we can check that 
		\begin{equation*}
		\frac{\partial v}{\partial t}=\log\left[ \frac{(A\omega_X+\dc v_t)^n}{\mu}\right], \;\,\text{while}\;\,    \frac{\partial \f}{\partial t}\leq\log\left[ \frac{(A\omega_X+\dc \f_t)^n}{\mu}\right],
		\end{equation*} and $v_0\geq \f_0$. Hence, by the classical maximum principle, we have $v(t,x)\geq\f(t,x)$ for $(t,x)\in [0,T]\times X$. Consequently, this provides an upper bound for $\f(t,x)$: $$\sup_X|\Phi|+\max(c,0)T+\sup_X\f_0.$$ \end{proof}
	We fix two positive constants $\alpha,\beta$ such that
	\[\frac{p^*}{2c(\f_0)}<\frac{1}{\alpha}<\frac{1}{\alpha-\beta}< T_{\max},\]
	and \[\theta+(\alpha-\beta)\chi\geq 0.\]
	We observe that the density $e^{-\alpha\f_0}f$ belongs to $L^q$ for $q>1$. Indeed, we choose $\delta>0$
	so small ($\alpha(p^*+\delta)<2c(\f_0)$) that $\frac{1}{q}=\frac{1}{p}+\frac{1}{p^*+\delta}$ with $q>1$. Applying H\"older's inequality and Skoda's theorem, we have
	\[\int_X e^{-\alpha q\f_0}f^qdV\leq \|f\|_{L^p}^q\left( \int_X e^{-\alpha (p^*+\delta)\f_0}dV\right)^{q/p^*+\delta}<+\infty. \]
	Thus, by~\cite{tosatti2010complex}, there exists a smooth $\theta$-psh function $u$ such that
	\[\beta^n(\theta+\dc u)^n=e^{\beta u-\alpha \f_0}fdV. \]
	\begin{proposition}\label{prop: lowbound}
		For $t\in (0,\alpha^{-1})$,
		\begin{equation*}
		(1-\alpha t)\f_0+\beta tu+n(t\log t -t) \leq \f_t.      
		\end{equation*}
		In particular, there exists a uniform constant $C>0$ such that \[ \f_0-C(t-t\log t)\leq \f_t,\quad \forall\, t\in (1,\alpha^{-1}).\] 
	\end{proposition}
	\begin{proof}
		The proof is identical to that of~\cite[Lemma 2.9]{guedj2017regularizing}. 
	\end{proof}

	Before establishing a lower bound for the solution $\f_t$, we first prove an upper bound for its time derivative, $\dot{\f}_t:=\frac{\partial\f}{\partial t}$. 
	\begin{proposition}\label{prop: upper1dot}
		For all $(t,x)\in(0,T]\times X$,
		\begin{equation}\label{eq: upperdot}
		\dot{\f}_t(x)\leq \frac{\f_t(x)-\f_0(x)}{t}+n.
		\end{equation}
	\end{proposition}
	\begin{proof}
		The proof is identical to that of~\cite[Proposition 3.1]{guedj2017regularizing} (also in~\cite{guedj2020pluripotential}). 
	\end{proof}
	
	We follow the approach in~\cite{di2017uniqueness} to derive the following uniform estimate for the complex parabolic Monge--Ampère equation. 
	
	\begin{theorem}\label{thm: C0estimate}
		Fix $\varepsilon>p^*\varepsilon_0$. For $t\in[\varepsilon,T]$,  we obtain the following estimate:
		\begin{equation*}
		\f_t\geq \left(1-\frac{bt}{T}\right)\psi_0-C,
		\end{equation*} where $b\in(a,1/2)$ and $C>0$ is a uniform constant. 
	\end{theorem}
	\begin{proof}
		Fixing $t\in[\varepsilon,T]$, it follows from Proposition~\ref{prop: upper1dot} that
		\begin{equation*}
		(\theta_t+\dc\f_t)^n=e^{\dot{\f}_t}\leq e^{n+\frac{\f_t-\f_0}{t}}fdV.
		\end{equation*}
		We set
		\begin{equation*}
		\psi_t:=\left(1-\frac{bt}{T} \right)\psi_0,
		\end{equation*} for $b\in (a,1/2)$ close to $a$.
		We recall that
		\[\theta_t\geq \left( 1-\frac{at}{S}\right)\theta,\]
		it then follows that $\psi_t$ is  $\delta\theta_t$-psh with $\delta\in (0,1)$ only depending on $\varepsilon_0$, $a$, $b$, $T$, $S$ (more precisely, $\delta=\frac{TS-bS\varepsilon_0}{TS-aT\varepsilon_0}$). Using similar arguments as in the proof of~\cite[Theorem 3.2]{di2017uniqueness}, we can bound the following quantity:
		\begin{equation}\label{eq: Lqmass}
		\int_X e^{\frac{q(\psi_t-\f_0)}{t}}f^qdV<+\infty, 
		\end{equation} for some $q>1$, in terms of $\|f\|_{L^p}$, $E_1$ and $E_2$. To establish this, we fix $\gamma>0$ small enough and choose $q>1$ such that
		\[\frac{1}{q}=\frac{1}{p}+\frac{1}{2p^*+\gamma}+\frac{1}{2p^*+\gamma}.\]
		By H\"older's inequality, we obtain
		\[\int_X e^{\frac{q(\psi_t-\f_0)}{t}}f^qdV\leq \|f\|_{L^p}^q \left(\int_X e^{\frac{(2p^*+\gamma)(\psi_0-\f_0)}{t}} dV \right)^{\frac{q}{2p^*+\gamma}}\left(\int_X e^{-\frac{(2p^*+\gamma)b\psi_0}{T}} dV\right)^{\frac{q}{2p^*+\gamma}}\]
		The second term on the right-hand side is finite due to the construction of $\psi_0$ in Lemma~\ref{lem: equisingular}. Also, since $\psi_0$ is less singular than $\f_0$, the third term is finite.
		
		From~\eqref{eq: Lqmass}, we apply Lemma~\ref{lem: C0_estimate}  with $A=1/t$ and $g=\f_0/t-n$ to obtain the desired estimate. It is important to note that our $\mathcal{C}^0$-estimate depends only on $n$, $\theta$, $q$, the fixed parameters $\varepsilon_0$, $\varepsilon$, $T$, $S$, and an upper bound for $E_1$ and $E_2$. 
	\end{proof}

	\begin{remark}
		When $\f_0$ is bounded or, more generally, has zero Lelong numbers, it was shown in~\cite{to2018regularizing} (generalizing the result of~\cite{guedj2017regularizing} in the K\"ahler context) that the estimate~\eqref{eq: upperdot} ensures a lower bound for $\f_t$ using Kolodziej--Nguyen's theorem~\cite{kolodziej2015weak}. Unfortunately, this method cannot be applied in more general cases, such as when $\f_0$ is more singular, for example, when it has a positive Lelong number. To analyze the singularities of the initial potential $\f_0$ in such cases, Guedj--Lu's approach~\cite{guedj2021quasi} could help.  
	\end{remark}

	\subsection{Laplacian estimate}
	We recall the following classical inequality
	\begin{lemma}\label{lem34}
		Let $\alpha$, $\beta$ be two positive (1,1)-forms. Then
		\begin{equation*}
		n\left( \frac{\alpha^n}{\beta^n}\right)^{\frac{1}{n}}\leq \textrm{tr}_\beta(\alpha)\leq n\left(\frac{\alpha^n}{\beta^n}\right)(\textrm{tr}_\alpha(\beta))^{n-1}.
		\end{equation*}
	\end{lemma}
	We define
	\begin{equation*}\label{def: Psi}
	\Psi_t:=\left(1-\frac{bt}{S}\right)\psi_0,
	\end{equation*} where $\p_0$ is defined in Lemma~\ref{lem: equisingular} with $\varepsilon_0>0$ fixed.
	\smallskip
	
	To establish the $\mathcal{C}^2$-estimate, it is necessary to derive a lower bound for $\dot{\f}_t=\frac{\partial \f}{\partial t}$. 
	
	\begin{proposition}\label{prop: 1dot}Fix $\varepsilon>p^*\varepsilon_0$.
		For $(t,x)\in(\varepsilon,T]\times X$, \begin{equation*}
		\dot{\f}_t(x)\geq n\log(t-\varepsilon)+A(\Psi_t-\f_t)-C
		\end{equation*} where $A, C>0$ are positive constants only depending on $\varepsilon$, $T$, $\|f\|_{L^p}$, and an upper bound for $E_1$ and $E_2$.
	\end{proposition}
	\begin{proof}
		The proof is almost identical to that of~\cite[Proposition 3.5]{di2017uniqueness}. The only difference is that we use Theorem~\ref{thm: C0estimate} instead of the corresponding result in~\cite[Theorem 3.2]{di2017uniqueness}. We include the proof for the reader's convenience.
		
		Since $\mu=fdV$ is a smooth volume form, the main result of Tosatti--Weinkove ~\cite{tosatti2010complex} ensures that there exists a constant $c_1$ and $\phi_1\in\PSH(X,\theta)\cap \mathcal{C}^\infty(X)$ such that
		\begin{equation*}(\theta+\dc\phi_1)^n=e^{c_1}\mu, \quad\sup_X\phi_1=0. \end{equation*}
		From~\cite[Theorems 2.2, 3.4]{guedj2021quasi}, it follows that $|c_1|+\|\phi_1\|_{L^\infty}\leq C$, where $C>0$ depends only on the semi-positivity and bigness of $\theta$, $n$, $dV_X$, $p$ and $\|f\|_p$.
		We define \[G(t,x):=\dot{\f}_t(x)+A(\f_t-\Psi_t)-\phi_1-n\log(t-\varepsilon)\] for a constant $A>0$ to be determined hereafter. Observe that $G$ achieves its minimum on $[\varepsilon,T]\times X$
		at some point $(t_0,x_0)\in (\varepsilon,T]\times (X\backslash\{\psi_0=-\infty\})$. In the following, all computations will be performed at this point. We compute
		\begin{equation*}\begin{split}
		\left(\frac{\partial }{\partial t}-\Delta_{t} \right)G=A\dot{\f}_t-\frac{n}{t-\varepsilon}+A\frac{b\psi_0}{S}-nA+A\tr_{\omega_t}(\theta_t+\dc \Psi_t)+\tr_{\omega_t}(\chi+\dc\phi_1).
		\end{split}
		\end{equation*}
		We observe that
		\begin{equation*}
		\begin{split}
		\theta_t+\dc\Psi_t&=\frac{t(b-a)}{S}\theta+  \left(1-\frac{bt}{S}\right)(\theta+\dc\psi_0)\\
		&\geq \frac{\varepsilon(b-a)}{S}\theta+\frac{1}{2}2\kappa\omega_X.
		\end{split}
		\end{equation*}
		We now choose $A>0$ so big that
		\[A( \theta_t+\dc\Psi_t)+\chi\geq \theta.\]
		Therefore
		\begin{equation}\label{eq: prop_dot}
		\left(\frac{\partial }{\partial t}-\Delta_{t} \right)G\geq A\dot{\f}_t-\frac{n}{t-\varepsilon}+A\frac{b\psi_0}{S}-nA+\tr_{\omega_t}(\theta+\dc\phi_1).
		\end{equation}
		On the other hand, Lemma~\ref{lem34} ensures that
		\[\tr_{\omega_t}(\theta+\dc\phi_1)\geq n\left( \frac{(\theta+\dc\phi_1)^n}{\omega_t^n}\right)^{1/n} =n e^{\frac{-\dot{\f}_t+c_1}{n}. }\]
		Using the elementary inequality $\gamma y-\log y\geq -C_{\gamma}$ for any small constant $\gamma>0$ and $y>0$, we observe that
		\[A\dot{\f}_t+n e^{\frac{-\dot{\f}_t+c_1}{n}}\geq   e^{\frac{-\dot{\f}_t}{n}-C_1}-C_2. \]
		Substituting this into~\eqref{eq: prop_dot}, it follows from the minimum principle that at $(t_0,x_0)$,
		\begin{equation*}
		\dot{\f}_t\geq -n\log\left( C_2+\frac{n}{t-\varepsilon}-\frac{Ab\psi_0}{S}+nA\right)-nC_1,
		\end{equation*}
		and hence,
		\begin{equation*}
		\begin{split}
		G(t_0,x_0) \geq -C_3-n\log\left( C_2(t_0-\varepsilon)+{n}-\frac{Ab(t_0-\varepsilon)\psi_0}{S}\right)-\frac{Abt_0(S-T)}{ST}\psi_0
		\end{split}
		\end{equation*}
		where we have used Theorem~\ref{thm: C0estimate}. Thus, we obtain a uniform lower bound for $G(t_0,x_0)$, and the desired lower bound follows.
	\end{proof}

	We are now in a position to establish the $\mathcal{C}^2$-estimate. We follow the computations of~\cite[Lemma 4.1]{tosatti2015evolution} (see also~\cite[Lemma 6.4]{to2018regularizing}), where they use the technical trick introduced by Phong and Sturm~\cite{phong2010dirichlet}.
	Recall that the measure $\mu$ is of the form
	\[\mu=e^{\psi^+-\psi^-}dV_X\] where $\psi^{\pm}$ are smooth $K\omega_X$-psh functions on $X$ for uniform constant $K>0$. For simplicity, we assume $K=1$ and normalize $\sup_X\psi^{\pm}=0$.
	\begin{theorem}\label{thm: C2estimate}
		Fix $\varepsilon>p^*\varepsilon_0$.
		For $(t,x)\in[\varepsilon,T]\times X$ we have the following bound
		\begin{equation*}
		(t-\varepsilon)\log\textrm{tr}_{\omega_X}(\omega_t)\leq -B\psi_0-C\psi^{-}+C
		\end{equation*} where $B$, $C$ are positive constants depending only on $\varepsilon$, $T$, $\|e^{-\psi^-}\|_{L^p}$, and an upper bound for $E_1$ and $E_2$.
	\end{theorem}
	
	\begin{proof} We follow the computations of~\cite{gill2011convergence,to2018regularizing} (which are due to the trick of Phong and Sturm~\cite{phong2010dirichlet}) with modification to deal with unbounded functions.
		The constant  $C$  denotes various uniform constants, which may differ throughout the argument.
		
		Consider 
		\begin{equation*}
		H:=(t-\varepsilon)\log\tr_{\omega_X}(\omega_t)-\gamma(u),\quad (t,x)\in [\varepsilon,T]\times X,
		\end{equation*} where $\gamma:\mathbb{R}\to\mathbb{R}$ is a smooth, concave, increasing function such that $\lim_{t\to+\infty}\gamma(t)=+\infty$, and 
		$$u(t,x):=\f_t(x)-\Psi_t(x)-\kappa\psi^{-}+1\geq 1,$$ as follows from Theorem~\ref{thm: C0estimate}, and $\psi_0, \psi^{-}\leq 0$.
		We will show that $H$ is uniformly bounded from above for an appropriate choice of $\gamma$. 
		
		We let $g$ denote the Riemann metric associated with $\omega_X$ and $\tilde{g}$ the one associated with $\omega_t:=\theta_t+\dc\f_t$. Since $H$ goes to $-\infty$ on the boundary of $X_0:=\{x\in X: \psi_0(x)>-\infty \}$, $H$ achieves its maximum on $[\varepsilon,T]\times X$ at some point $(t_0,x_0)\in (\varepsilon,T]\times X_0$. 
		At this maximum point, we use the following local coordinate systems due to Guan and Li~\cite[Lemma 2.1, (2.19)]{guan2010complex}:
		\begin{equation*}
		g_{i\bar{j}}=\delta_{ij},\; \frac{\partial g_{i\bar{i}}}{\partial z_j}=0\; \text{and}\; \tilde{g}_{i\bar{j}}\; \text{is diagonal}.
		\end{equation*}
		Following the computations in~\cite[Eq. (3.20)]{to2018regularizing}, we have
		\begin{equation}\label{eq: delta}
		\begin{split}
		\Delta_t\tr_{\omega_X}(\om_t)&\geq \sum_{i,j}\Tilde{g}^{i\bar{i}}\Tilde{g}^{j\bar{j}}\Tilde{g}_{i\bar{j}j}\Tilde{g}_{j\bar{i}\bar{j}}-\tr_{\omega_X}\textrm{Ric}({\omega}_t)
		-C_1\tr_{\omega_X}(\om_t)\tr_{\om_t}(\omega_X).
		\end{split}
		\end{equation}
		From standard arguments as in~\cite[Eq. (4.5)]{guedj2021quasi},  we obtain
		\begin{equation}\label{eq: error}
		\begin{split}
		\frac{{|\partial\tr_{\omega_X}(\om_t)|_{\tom_t}^2}}{(\tr_{\omega_X}(\om_t))^2}&\leq \frac{1}{\tr_{\omega_X}(\om_t)}\left(   \sum_{i,j}\Tilde{g}^{i\bar{i}}\Tilde{g}^{j\bar{j}}\Tilde{g}_{i\bar{j}j}\Tilde{g}_{j\bar{i}\bar{j}} \right)+C\frac{\tr_{\om_t}(\omega_X)}{(\tr_{\omega_X}(\om_t))^2}\\
		&\quad +\frac{2}{(\tr_{\omega_X}(\om_t))^2}\textrm{Re}\sum_{i,j,k} \Tilde{g}^{i\bar{i}} T_{ij\bar{j}}\Tilde{g}_{k\bar{i} \bar{k}},\end{split}
		\end{equation} where $T_{ij\bar{j}}:=\tilde{g}_{j\bar{j}i}-\tilde{g}_{i\Bar{j}j}$ is the torsion term corresponding to $\theta_{t}$ which is controlled: $|T_{ij\bar{j}}|\leq C$.
		Now at the point $(t_0,x_0)$, we have $\partial_{\Bar{i}}H=0$, hence
		\begin{equation*}
		(t-\varepsilon)\sum_k \Tilde{g}_{k\bar{k}\Bar{i}}=\tr_{\omega_X}(\om_t)\gamma'(u)u_{\bar{i}}.
		\end{equation*}
		Cauchy--Schwarz's inequality yields
		\begin{equation*}
		\left|\frac{2}{(\tr_{\omega_X}(\om_t))^2}\textrm{Re}\sum_{i,j,k} \Tilde{g}^{i\bar{i}} T_{ij\bar{j}}\Tilde{g}_{k\bar{k}\bar{i}} \right|\leq C\frac{\gamma'(u)(t_0-\varepsilon)}{-\gamma''(u)}\frac{\tr_{\omega_t}(\omega_X)}{(\tr_{\omega_X}(\omega_t))^2}+\frac{-\gamma''(u)}{t_0-\varepsilon}|\partial u|^2_{\omega_t},
		\end{equation*} hence
		\begin{equation*}
		\left|\frac{2}{(\tr_{\omega_X}(\om_t))^2}\textrm{Re}\sum_{i,j,k} \Tilde{g}^{i\bar{i}} T_{ij\bar{j}}\Tilde{g}_{k\bar{i}\bar{k}} \right|\leq C\left(\frac{\gamma'(u)T}{-\gamma''(u)}+1\right)\frac{\tr_{\omega_t}(\omega_X)}{(\tr_{\omega_X}(\omega_t))^2}+\frac{-\gamma''(u)}{t_0-\varepsilon}|\partial u|^2_{\omega_t},
		\end{equation*} using that $|\Tilde g_{k\bar{k}\bar{i}}-\Tilde g_{k\bar{i}\bar{k}}|\leq C$. From this, the inequality~\eqref{eq: error} becomes
		\begin{equation}\label{eq: grad}
		\begin{split}
		\frac{{|\partial\tr_{\omega_X}(\om_t)|_{\om_t}^2}}{(\tr_{\omega_X}(\om_t))^2}&\leq \frac{1}{\tr_{\omega_X}(\om_t)}\left(   \sum_{i,j}\Tilde{g}^{i\bar{i}}\Tilde{g}^{j\bar{j}}\Tilde{g}_{i\bar{j}j}\Tilde{g}_{j\bar{i}\bar{j}} \right)\\
		&\quad+C\left(\frac{\gamma'(u)T}{-\gamma''(u)}+2\right)\frac{\tr_{\omega_t}(\omega_X)}{(\tr_{\omega_X}(\omega_t))^2}+\frac{-\gamma''(u)}{t_0-\varepsilon}|\partial u|^2_{\omega_t}.\end{split}
		\end{equation}
		Set $\alpha:=\tr_{\omega_X}(\omega_t)$. We compute
		\begin{equation*}
		\begin{split}
		\dot{\alpha}&=\tr_{\omega_X}(\chi)-\tr_{\omega_X}\textrm{Ric}(\omega_t)-\tr_{\omega_X}\dc (\psi^+-\psi^{-})+ \tr_{\omega_X} (\textrm{Ric}(\omega_X))\\
		&\leq \tr_{\omega_X}(C_1\omega_X+\dc\psi^-) -\tr_{\omega_X}\textrm{Ric}(\omega_t)
		\end{split}
		\end{equation*} where we have used the fact that $\tr_{\omega_X}(\chi)$ is bounded from above, together with the trivial inequality $n\leq\tr_{\omega_X}(\omega_t)\tr_{\omega_t}(\omega_X)$.
		Combining this  with~\eqref{eq: delta} and~\eqref{eq: grad}, we infer that
		\begin{equation}
		\begin{split}
		\frac{\dot{\alpha}}{\alpha}-\Delta_{t}\log\alpha&=\frac{\dot{\alpha}}{\alpha}-\frac{\Delta_{t}\alpha}{\alpha}+\frac{|\partial\alpha|^2_{\omega_t}}{\alpha^2} \\
		&\leq \frac{\tr_{\omega_t}(C_1\omega_X+\dc\psi^{-})}{\alpha}+ C\left(\frac{\gamma'(u)T}{-\gamma''(u)}+2\right)\frac{\tr_{\omega_t}(\omega_X)}{\alpha^2}+\frac{-\gamma''(u)}{t_0-\varepsilon}|\partial u|^2_{\omega_t}.
		\end{split}
		\end{equation}
		From this, at the maximum point $(t_0,x_0)$,
		\begin{equation}
		\begin{split}\label{eq: maximum}
		0\leq \left(\frac{\partial}{\partial t}-\Delta_{t} \right)H&=\log\alpha+ (t-\varepsilon)\left( \frac{\dot{\alpha}}{\alpha}-\Delta_{t}\log\alpha \right) \\
		&\quad -\gamma'(u)\dot{u} +\gamma'(u)\Delta_t u+\gamma''(u)|\partial u|^2_{\omega_t}\\
		&\leq \log\alpha+ \frac{C_3\tr_{\omega_t}(\omega_X+\dc\psi^{-})}{\alpha}+ C_4\left(\frac{\gamma'(u)T}{-\gamma''(u)}+2\right)\frac{\tr_{\omega_t}(\omega_X)}{\alpha^2}\\
		&\quad-\gamma'(u)\dot{\f}_t+\gamma'(u)\dot{\Psi}_t  +  \gamma'(u)\Delta_{\omega_t}(\f_t-\Psi_t-\kappa\psi^{-})
		,
		\end{split} 
		\end{equation} with $C_3$, $C_4>0$ under control.
		Moreover, since $\theta_t\geq \left(1-\frac{at}{S}\right)\theta$ hence 
		$$\theta_t+\dc\Psi_t\geq \left(1-\frac{bt}{S}\right)2\kappa\omega_X.$$ Thus we obtain
		\begin{equation}\label{eq: dpsi}
		\Delta_t(\f_t-\Psi_t)\leq n-\kappa\tr_{\omega_t}(\omega_X).
		\end{equation}
		Substituting~\eqref{eq: dpsi} into~\eqref{eq: maximum}, we obtain
		\begin{equation*}
		\begin{split}
		0&\leq\log\alpha +\frac{C_3\tr_{\omega_t}(\omega_X+\dc\psi^{-})}{\alpha}- \gamma'(u)(n-\kappa\tr_{\omega_t}(\omega_X+\dc\psi^{-}))
		\\
		&\quad-\gamma'(u)\dot{\f}_t-\gamma'(u)\frac{b\psi_0}{S}+ C_4\left(\frac{\gamma'(u)T}{-\gamma''(u)}+2\right)\frac{\tr_{\omega_t}(\omega_X)}{(\tr_{\omega_X}(\omega_t))^2}+C_5.
		\end{split}
		\end{equation*}
		We now choose the function $\gamma$ to obtain a simplified formulation. We set
		\begin{equation*}
		\gamma(u):=\frac{C_3+3}{\min(\kappa,1)} u+\log (u).
		\end{equation*} 
		Since $u\geq 1$ we have
		\begin{equation*}
		\frac{C_3+3}{\min(\kappa,1)} \leq\gamma'(u)\leq 1+\frac{C_3+3}{\min(\kappa,1)},\qquad \frac{\gamma'(u)T}{-\gamma''(u)}+2\leq C_5 u^2.
		\end{equation*}
		Using $\tr_{\omega_X}(\omega_X+\dc\psi^{-})\leq \tr_{\omega_t}(\omega_X+\dc\psi^{-})\tr_{\omega_X}(\omega_t)$ we obtain
		\begin{equation}\label{eq: max}\begin{split}
		0&\leq \log\alpha- \gamma'(u)\dot{\f}_t-\gamma'(u)\frac{b\psi_0}{S}-3\tr_{\omega_t}(\omega_X)
		+C_6(u^2+1)\frac{\tr_{\omega_t}(\omega_X)}{\alpha^2}.
		\end{split}
		\end{equation}
		If at the point $(t_0,x_0)$, we have $\alpha^2\leq C_6(u^2+1)$ then
		\begin{equation*}
		H(t_0,x_0)\leq T\log \sqrt{C_6(u^2+1)}-\gamma(u)\leq C_7,
		\end{equation*} we are done. Otherwise, we assume that, at $(t_0,x_0)$, $\alpha^2\geq C_6(u^2+1)$. Applying Lemma~\ref{lem34}, we obtain
		\[\log\alpha=\log\tr_{\omega_X}(\omega_t)\leq (n-1)\log \tr_{\omega_t}(\omega_X)+\log n+\dot{\f}_t-\psi^{-} \] using that $\sup_X\psi^{+}=0$. 
		Plugging this into~\eqref{eq: max}, we obtain
		\begin{equation*}
		0\leq  C_5+(n-1)\log \tr_{\omega_t}(\omega_X)-2\tr_{\omega_t}(\omega_X)- (\gamma'(u)-1)\dot{\f}_t-\gamma'(u)\frac{b\psi_0}{S}-\psi^{-},
		\end{equation*}
		or equivalently,
		\begin{equation}\label{eq: trace}
		\tr_{\omega_t}(\omega_X)\leq C_8-(\gamma'(u)-1)\dot{\f}_t-\gamma'(u)\frac{b\psi_0}{S}-\psi^{-}
		\end{equation} since $(n-1)\log y -2y\leq -y+O(1)$ for $y>0$.
		In particular, we have
		\begin{equation}\label{eq: dotf}
		\dot{\f}_t\leq \frac{C_5}{\gamma'(u)-1}-\frac{\gamma'(u)}{\gamma'(u)-1}\frac{b\psi_0}{S}\leq \frac{C_5}{A-1}-\frac{bA\psi_0}{(A-1)S}-\frac{\psi^{-}}{A-1}
		\end{equation} at $(t_0,x_0)$ since $\tr_{\omega_t}(\omega_X)\geq 0$ and $A\leq \gamma'(u)\leq A+1$ with $A=:\frac{C_3+3}{\min(\kappa,1)}$. 
		It follows from Lemma~\ref{lem34} that
		\[\tr_{\omega_t}(\omega_X)\geq n\exp\left(\frac{-\dot{\f}_t+\psi^{-}}{n}\right) .\]
		Plugging this into~\eqref{eq: trace}, we obtain
		\begin{equation*}
		\tr_{\omega_t}(\omega_X)\leq C_9-\gamma'(u)\frac{b\psi_0}{S}-\gamma'(u)\psi^{-}\leq C_9-\frac{(A+1)b\psi_0}{S}-(A+1)\psi^{-}
		\end{equation*} with $C_9>0$ under control,
		since $e^y- Dy \geq -C$ for $y\in \mathbb{R}$, $D>0$ we apply with $y=\frac{-\dot{\f}_t+\psi^{-}}{n}$. Again Lemma~\ref{lem34} yields
		\begin{equation*}
		\log\alpha\leq (n-1)\log\left(C_9-\frac{b(A+1)\psi_0}{S} -(A+1)\psi^{-}\right)+\log n+\dot{\f}_t-\psi^{-}.
		\end{equation*}
		Combining this together with~\eqref{eq: dotf},
		we have at $(t_0,x_0)$
		\begin{equation*}
		\begin{split}
		H&\leq C_{10}-A\left[\f_t-\left( 1-\frac{bt}{S}-\frac{b(t-\varepsilon)}{(A-1)S} \right)\psi_0 \right]+\left(A\kappa-1-\frac{1}{A-1} \right)\psi^{-}\\
		&\quad+(t-\varepsilon) (n-1)\log\left(C_9-\frac{b(A+1)\psi_0}{S}-(A+1)\psi^{-} \right).
		\end{split}
		\end{equation*}
		Up to increasing $A>0$ if necessary, so that
		\[\eta:=\frac{b\varepsilon}{T}-\frac{b\varepsilon}{S}-\frac{bT}{(A-1)S}>0,\]
		and since $\psi_0\leq 0$, we obtain, at $(t_0,x_0)$,
		\begin{equation*}
		\begin{split}
		H&\leq C_{10}-A\left[\f_t-\left(1-\frac{bt}{T}\right)\psi_0\right]+A\eta\psi_0+A\kappa/2\psi^{-}\\
		&\quad+(t-\varepsilon) (n-1)\log\left(C_9-\frac{b(A+1)\psi_0}{S}-(A+1)\psi^{-} \right).
		\end{split}   
		\end{equation*} The second term is uniformly bounded from above by Theorem~\ref{thm: C0estimate}. Since $-\gamma y+\log y$ is bounded from above for $y>0$, we conclude that $H$ achieves a uniform bound at $(t_0,x_0)$. This completes the proof.
		
	\end{proof}
	
	\subsection{Estimates near the zero time} Recall that there exists a $\theta$-psh function $\rho$ with analytic singularities such that $\sup_X\rho=0$ and $$\theta+\dc\rho\geq 3\delta_0\omega_X$$ for some $\delta_0>0$. The main result of Tosatti--Weinkove ~\cite{tosatti2010complex} ensures that there exists a constant $c_1$ and  $\phi_1\in\PSH(X,\theta)\cap \mathcal{C}^\infty(X)$ such that 
	\[(\theta+\dc\phi_1)^n=e^{c_1}d\mu, \quad\sup_X\phi_1=0. \]
	\begin{proposition}\label{prop: zero1}
		Assume that $\psi_1$, $\psi_2$ are two smooth $\omega_X$-psh functions satisfying  \[\dot{\f}_0\geq C_1\psi_1,\quad \f_0\geq \frac{1}{2}(\rho+\delta_0\psi_2) \]
		for some constants $C_1>0$. Fix $T_1\in (0,T_{\max})$ such that $\theta_{t}> \frac{1}{2}\theta$ for all $t\in[0,T_1]$. Then there exists a uniform constant $C_2>0$ only depending on $C_1$, $\delta_0$, $T_1$ and $\sup_X|\phi_1|$ such that \[\dot{\f}_t\geq C_2(\rho+\delta_0
		\psi_2+1)+C_1\psi_1,\; \forall\, t\in [0,T_1]. \]  
	\end{proposition}
	
	\begin{proof}The proof is identical to that of Proposition~\ref{prop: 1dot}. We consider
		\[H(t,x):=\dot{\f}_t-C_1\psi_1+A\left(\f_t-\frac{1}{2}(\rho+\delta_0\psi_2)\right)-\phi_1,\]
		for $A>0$ to be chosen later. We observe that $H$ achieves its minimum at some point $(t_0,x_0)\in [0,T_1]\times X$. If $t_0=0$, we are done by assumptions.
		Otherwise, by the minimum principle, we have at $(t_0,x_0)$,
		\[0\geq \left( \frac{\partial}{\partial t}-\Delta_t\right)H\geq -An+A\dot{\f}_t +\left( -C_1+A\delta_0\right)\tr_{\omega_t}(\omega_X)+\tr_{\omega_t}(\dc\phi_1)\]
		using $\theta_t+\dc\frac{1}{2}(\rho+\delta_0\psi_2)\geq \delta_0\omega_X$. Now, we choose $A=\delta_0(C_1+1)$, thus
		\[\tr_{\omega_t}(\omega_X+\dc\phi_1)\geq n\left( \frac{(\theta+\dc\phi_1)^n}{\omega_t^n}\right)^{1/n} =n e^{\frac{-\dot{\f}_t+c_1}{n}}\]
		using  Lemma~\ref{lem34}. Together with the inequality $e^y\geq By-C_B$, we obtain a uniform lower bound for $\dot{\f}_t$ at $(t_0,x_0)$. On the other hand, by Proposition~\ref{prop: lowbound} we see that $\f_t\geq \f_0-c(t)$, so $$\f_t\geq \frac{1}{2}(\rho+\delta_0\psi_2)-c(t),$$ where $c(t)\to 0$ as $t\to 0$. The lower bound for $H(t_0,x_0)$ thus follows, finishing the proof.
	\end{proof}

	\begin{proposition}\label{prop: zero2}
		Assume that $\psi_1$, $\psi_2$ are two smooth $\omega_X$-psh functions satisfying  \[\Delta_{\omega_X}\f_0\leq e^{-C_1\psi_1},\quad \f_0\geq \frac{1}{2}(\rho+\delta_0\psi_2)\]
		for some constants $C_1>0$. Fix $T_1\in (0,T_{\max})$ such that $\theta_{t}> \frac{1}{2}\theta$ for all $t\in[0,T_1]$. Then there exist uniform constants $C_2>0$, $C_3>0$ only depending on $C_1$, $\delta_0$ and $T_1$ such that
		\begin{equation*}
		\tr_{\omega_X}(\omega_t)\leq C_3e^{-C_1\psi_1-C_2(\rho+\delta_0\psi_2+\delta_0\psi^{-}}),\; \forall\, t\in [0,T_1].
		\end{equation*}
	\end{proposition}
	\begin{proof} 
		Consider the function
		\begin{equation*}
		H(t,\cdot)=\log\tr_{\omega_X}(\omega_t)+C_1\psi_1-\gamma (u)
		\end{equation*} where $\gamma:\mathbb{R}\to\mathbb{R}$ is a smooth concave increasing function such that $\lim_{t\to+\infty}\gamma(t)=+\infty$, and 
		$$u(t,x):=\f_t(x)-\frac{1}{2}(\rho(x)+ \delta_0\psi_2(x))+\delta_0\psi^{-}(x)+1.$$
		We suppose that $H$ achieves its maximum at a point $(t_0,x_0)\in [0,T_1]\times X$, with $x_0\in  \{\rho>-\infty\}$. If $t_0=0$, then $H(0,\cdot)\leq \log n-\gamma(1)$. Otherwise, assume $t_0>0$. We proceed by computing at this point.
		By the maximum principle and the arguments in Theorem~\ref{thm: C2estimate}, we have
		\begin{equation}\begin{split}
		0\leq \left(\frac{\partial}{\partial t}-\Delta_t\right)H&\leq \frac{C\tr_{\omega_t}(\omega_X+\dc\psi^{-})}{\tr_{\omega_X}(\omega_t)}- \gamma'(u)(n-\delta_0\tr_{\omega_t}(\omega_X+\dc\psi^{-}))+C
		\\
		&\quad-C_1\tr_{\omega_t}(\dc\psi_1)-\gamma'(u)\dot{\f}_t+ C\left(\frac{\gamma'(u)}{-\gamma''(u)}+2\right)\frac{\tr_{\omega_t}(\omega_X)}{(\tr_{\omega_X}(\omega_t))^2}.
		\end{split}
		\end{equation}Here, we use $\theta_t+\dc\frac{1}{2}(\rho+\delta_0\psi_2)\geq \delta_0\omega_X$. We set
		\begin{equation*}
		\gamma(u):=\frac{C+C_1+3}{\min(\kappa,1)} u+\ln (u).
		\end{equation*} We then proceed in the same way as in the proof of Theorem~\ref{thm: C2estimate} to obtain the uniform upper bound for $H(t_0,x_0)$. This finishes the proof.
	\end{proof}

	\section{Degenerate Monge--Amp\`ere flows}\label{sect: existence}
	\subsection{Proof of Theorem~\ref{thmB}}
	By Demailly's regularization theorem (Theorem~\ref{thm: dem}), we can find two sequences $\psi_j^\pm\in\mathcal{C}^\infty(X)$ such that
	\begin{itemize}
		\item $\psi_j^{\pm}$ decreases pointwise to $\psi^\pm$ on $X$ and the convergence is in $\mathcal{C}^\infty_{\rm loc}(U)$;
		\item $\dc\psi^\pm\geq -\omega_X$.
	\end{itemize} We note that $|\sup_X\psi^\pm_j|$ is uniformly bounded, and for all $j$, \[ \|e^{-\psi^-_j}\|_{L^p}\leq \|e^{-\psi^-}\|_{L^p}.\]
	Thanks to Demailly's regularization theorem again, we can find a smooth sequence $(\f_{0,j})$ of strictly $(\theta+2^{-j}\omega_X)$-psh functions decreasing towards $\f_0$. We set $\theta_{t,j}=\theta_t+2^{-j}\omega_X$ and $\mu_j=e^{\psi^+_j-\psi^-_j}$.
	It follows from~\cite[Theorem 1.2]{tosatti2015evolution} (see also~\cite{to2018regularizing}) that there exists a unique function $\f_j\in \mathcal{C}^\infty([0,T)\times X)$ such that 
	\begin{equation}\label{eq: pcmae-smooth}
	\begin{cases}
	\dfrac{\partial\f_{t,j}}{\partial t}=\log\left[ \dfrac{(\theta_{t,j}+\dc\f_{t,j})^n}{\mu_j}\right]\\
	\f_j|_{t=0}=\f_{0,j}.
	\end{cases}
	\end{equation} 
	It follows from the maximum principle that the sequence $\f_{t,j}$ decreases with respect to $j$. Moreover, Proposition~\ref{prop: C0upper} ensures that $\sup_X\f_{t,j}$ is uniformly bounded from above. By Proposition~\ref{prop: lowbound}, as $j\to+\infty$, the family $\f_{t,j}$ decreases to $\f_t$ which is a well-defined $\theta_t$-psh function on $X$. Following the same arguments as in~\cite[Section 4.1]{to2018regularizing}, we conclude that $\f_t\to\f_0$ in $L^1(X)$ as $t\to 0^+$.
	
	Next, we study the partial regularity of $\f_t$ for small $t$. We fix $\varepsilon_0>0$ and $\varepsilon>p^*\varepsilon_0$.  Let $T$ and $S$ be as defined in Section~\ref{sect: notation}. Let $\rho$ be a $\theta$-psh function with analytic singularities along $D$ such that $\theta+\dc\rho$ dominates a Hermitian form, where $D:=\{\rho=-\infty\}$.
	By Lemma~\ref{lem: equisingular}, there is a function $\psi_0\in\PSH(X,\theta)\cap \mathcal{C}^\infty(X\setminus (D\cup E_c(\f_0)))$, where $c=c(\varepsilon_0)>0$, such that 
	\begin{equation*}
	\int_X e^{\frac{2(\psi_0-\f_0)}{\varepsilon_0}}dV_X<+\infty.
	\end{equation*}
	We assume w.l.o.g that $\psi_0\leq 0$. Since $\frac{p^*}{2c(\f_0)}<T$ and $\psi_0$ is less singular than $\f_0$, we also have
	\begin{equation*}
	\int_X e^{\frac{-p^*\psi_0}{T}}dV_X <+\infty.
	\end{equation*} We note that since $\f_0$ is a decreasing limit of a smooth sequence $\f_{0,j}$, the corresponding constants for $\f_{0,j}$ are uniformly bounded (in $j$), and we can pass to the limit as $j\to+\infty$.

	Recall that $\psi^\pm$ are smooth (merely locally bounded) in a Zariski open set $U\subset X\backslash D$.
	We will show that $\f_t$ is smooth on $U\setminus  E_c(\f_0)$ for each $t>\varepsilon$. Let $K$ be an arbitrarily compact subset of $U\setminus  E_c(\f_0)$. It follows from Proposition~\ref{prop: C0upper}, Theorem~\ref{thm: C0estimate}, and the remark above that
	\begin{equation*}
	\sup_{[\varepsilon,T]\times K}|\f_j|\leq C(\varepsilon,T,K).
	\end{equation*}
	Next, Proposition~\ref{prop: 1dot} yields 
	\begin{equation*}
	\sup_{[\varepsilon,T]\times K}|\dot{\f}_j|\leq C(\varepsilon,T,K).
	\end{equation*} Moreover, by Theorem~\ref{thm: C2estimate}, we also obtain a uniform bound for $\Delta\f_t^j$:
	\begin{equation*}
	\sup_{[\varepsilon,T]\times K}|\Delta\f_j|\leq C(\varepsilon,T,K).
	\end{equation*} Using the complex parabolic Evans–Krylov–Trudinger theory, together with parabolic Schauder's estimates (see, e.g., \cite[Theorem 4.1.4]{boucksom2013regularizing}), we derive higher-order estimates for $\f_j$ on $[\varepsilon, T]\times K$ :
	\begin{equation*}
	\|\f_j\|_{\mathcal{C}^k([\varepsilon,T]\times K)}\leq C(\varepsilon,T,K,k).
	\end{equation*}
	This ensures that $\f_j$ is relatively compact in $\mathcal{C}^\infty([\varepsilon,T]\times (U\setminus  E_c(\f_0))$ since $K$ was taken arbitrarily. 
	By passing to the limit in~\eqref{eq: pcmae-smooth}, we deduce that $\f$ satisfies~\eqref{pcmae} in the classical sense on $[\varepsilon,T]\times \Omega_\varepsilon$ with $\Omega_\varepsilon=U\setminus  E_{c(\varepsilon)}(\f_0)$.

	\subsection{Uniqueness}\label{sect: unique} 
	We now follow the argument in~\cite{guedj2017regularizing} to prove that the solution $\f$ to the equation~\eqref{pcmae} constructed in the previous part is the unique maximal solution in the following sense:
	\begin{proposition}
		Let $\psi_t$ be a weak solution to the equation~\eqref{pcmae} with initial data $\f_0$. Then $\psi_t\leq \f_t$ for all $t\in (0,\tmax)$. 
	\end{proposition}
	\begin{proof}
		By construction in the previous paragraph, $\f_{t,j}$ are smooth $(\theta_t+2^{-j}\omega_X)$-psh functions decreasing pointwise to $\f_t$. It thus suffices to show that $\psi_t\leq \f_{t,j}$ for all fixed $j$.
		
		Fix $0<T<\tmax$ and $2^{-j}>\varepsilon>\delta>0$. We let $U_\varepsilon\subset X$ denote the Zariski open set in which $\psi_{t+\varepsilon}$ is smooth. We can find a $\omega_X$-psh function $\phi$ with analytic singularities along $X\setminus U_\varepsilon$; see, e.g.,~\cite{demailly2004numerical}. We apply the maximum principle to the function $H:=\psi_{t+\varepsilon}-\f_{t+\varepsilon,j}+\delta\phi$. Suppose that $H$ achieves its maximum on $[0,T-\varepsilon]\times X$ at $(t_\varepsilon,x_\varepsilon)$ with $t_\varepsilon>0$. Note that $x_\varepsilon\in U_\varepsilon$. We thus have
		\begin{equation*}
		0\leq \frac{\partial}{\partial t}H\leq \log\left[ \frac{(\theta_{t+\varepsilon}+\dc\f_{t+\varepsilon,j}-\delta\dc\phi)^n}{(\theta_{t+\varepsilon}+2^{-j}\omega_X+\dc\f_{t+\varepsilon,j})^n}\right]<0
		\end{equation*} using that $-\dc\phi\leq \omega_X$, which is a contradiction. Letting $\delta\searrow 0$, we obtain
		\[\psi_{t+\varepsilon}(x)-\f_{t+\varepsilon,j}(x)\leq \sup_X(\psi_{\varepsilon}-\f_{\varepsilon,j}).\] Moreover, since $(\varepsilon,x)\mapsto\f_{\varepsilon,j}(x)$ is continuous, it follows from Hartogs' lemma (cf.~\cite[Proposition 8.4]{guedj2017degenerate}) that \[\sup_X(\psi_{\varepsilon}-\f_{\varepsilon,j})\xrightarrow{\varepsilon\to 0}\sup_X(\f_0-\f_{0,j})\leq 0.\] Letting $\varepsilon\to 0$, the desired inequality follows.
	\end{proof}
	The uniqueness we have just shown is referred to as "maximally stretched" by P. Topping in the context of Riemann surfaces; see~\cite[Remark 1.9]{topping2010ricci}.
	
	\subsection{Short time behavior}
	In this subsection, we study the behavior of the solution to the degenerate Monge--Amp\`ere flow in a short time. We show that the flow $\f_t$ starting from a current with positive Lelong numbers also has positive Lelong numbers for a sufficiently short time.  This result follows almost verbatim from the K\"ahler case, as discussed in~\cite[Section 4.2]{di2017uniqueness}. 
	\begin{theorem}
		If $\f_0$ has positive Lelong numbers, then
		\[E_c(\f_0)\subset E_{c(t)}(\f_t),\qquad c(t)=c-2nt.\]
		In particular, the maximal solution $\f_t$ has positive Lelong numbers for any $t<1/2nc(\f_0)$.
	\end{theorem}
	\begin{proof}
		The proof is identical to that of~\cite[Theorem 4.5]{di2017uniqueness}. We give a sketch of the proof here. Fix $x_0\in E_c(\f_0)$. We can find a cutoff function $\chi\in\mathcal{C}^\infty(X)$ with support near $x_0$ and $\chi=1$ on a neighborhood of $x_0$. Define $\phi:=\chi(x)c\log\|x-x_0\|$, which is $B\omega_X$-psh, and $e^{2\phi/c}\in\mathcal{C}^\infty(X)$. Since $x_0\in E_c(\f_0)$ we can choose $\phi$ so that $\phi\geq \f_0$ by adding a positive constant. By Lemma~\ref{lem: lem44}, we obtain \[\f_t\leq (1-2nt/c)\phi+Ct, \] which implies $\nu(\f_t,x_0)\geq c-2nt$. If $t<1/2nc(\f_0)$, then by Skoda's integrability theorem, $e^{-2\f_0/c}$ is not integrable for $2nt<c<1/c(\f_0)$. Therefore, $E_c(\f_0)$ is not empty, neither is $E_{c(t)}(\f_t)$ for sufficiently small $t>0$.
	\end{proof}
	\begin{lemma}\label{lem: lem44}
		Assume that $\phi\in\PSH(X,\omega_X)$ satisfies $e^{\gamma\phi}\in\mathcal{C}^\infty(X)$ for some constant $\gamma>0$, and $0\geq\psi^{\pm}\geq \phi\geq\f_0$.  Then, there exists a positive constant $C$ depending on an upper bound for $\dc e^{\gamma\phi }$ such that
		\[\f(t)\leq (1-(n\gamma+1)t)\phi+Ct, \quad\forall\,t\in [0,1/n\gamma].\]
	\end{lemma}
	\begin{proof}
		Assume that 
		$\theta_t\leq \omega_X$ for $t\in[0,1/(n\gamma+1)]$. As argued in~\cite[Lemma 4.4]{di2017uniqueness}, we can assume that $\phi$ is smooth and work with the approximants $\f_{t,j}$ instead. We choose $C>0$ depending only on an upper bound for $\dc e^{\gamma\phi}$, such that $\dc\phi\leq Ce^{-\gamma\phi}\omega_X$. Consider the function
		\[\phi_t:=(1-(n\gamma+1)t)\phi+t\log (2^nC^n).\]
		We observe that
		\[0\leq \omega_X+\dc\phi\leq 2Ce^{-\gamma\phi}\omega_X,\]
		hence \[(\omega_X+\dc\phi_t)^n\leq (2C)^ne^{-n\gamma\phi}\omega_X^n\leq e^{\dot{\phi}_t+\psi^+-\psi^-}\omega_X^n.\]
		Therefore, $\phi_t$ is a supersolution to the parabolic equation \[(\omega_X+\dc u_t)^n=e^{\dot{u}_t+\psi^+-\psi^-}\omega_X^n\] while $\f_{t,j}$ is a subsolution. By the classical maximum principle, it follows that $\f_{t,j}\leq \phi_t$ for any fixed $j$. This completes the proof.
	\end{proof}
	
	\subsection{Convergence at time zero}\label{sect: conv_zero}
	We study in this part the convergence at zero of the degenerate complex Monge--Amp\`ere flow. 
	
	We recall the quasi-monotone convergence in the sense of Guedj--Trusiani~\cite{guedj2022monotone}: $\f_j$ converges quasi-monotonically to $\f$ if $P_\theta(\inf_{\ell\geq j}\f_\ell)$ is a sequence of $\theta$-psh functions that increases to $\f$.
	\begin{theorem}
		The flow $\f_t$ converges quasi-monotonically to $\f_0$ as $t\to 0^{+}$.
	\end{theorem}
	\begin{proof}
		By Proposition~\ref{prop: lowbound}, we have that for small $t>0$,
		\[ \f_t\geq \f_0-C(t-t\log t).\] It follows that
		\[P_{\theta}\left(\inf_{0<s\leq t}\f_s\right)\geq \f_0-C(t-t\log t),\]  which completes the proof.
	\end{proof}
	\begin{theorem}
		Assume that $\f_0$ is continuous in an open set $U\subset X$. Then $\f_t$ converges to $\f_0$ in $L_{\rm loc}^\infty(U)$.
	\end{theorem}
	\begin{proof}The proof closely follows the arguments in the K\"ahler case~\cite{di2017uniqueness}.
		Without loss of generality, we assume that $\f_t\leq 0$. By Proposition~\ref{prop: lowbound}, there exists a uniform constant $C>0$ and a small time $t_0$ such that for $0\leq s<t\leq t_0$,
		\[\f_s-C(t-s)\log(t-s)-C(t-s)\leq\f_t.  \]   Set $u_t:=\f_t+(C+\log 4)t-Ct\log t$. Substituting $s=t/2$, we deduce that $u_t\geq u_{t/2}$, hence the sequence $u_{t_02^{-j}}$ decreases to $u_0=\f_0$. The conclusion follows from Dini's theorem.
	\end{proof}
	
	We also have the same result as in the K\"ahler case~\cite[Theorem 6.3]{di2017uniqueness}. We assume that $\theta$ is a big form and that $f=e^{\psi^+-\psi^-}\in L^p$, for some $p>1$, where $\psi^{\pm}$ are quasi-psh functions. Assume moreover that $\psi^-\in L^\infty_{\rm loc}(X\setminus D)$ for some closed set $D\subset X$. 
	It follows from~\cite[Theorem 4.1]{guedj2021quasi} that there exists a bounded  $\theta$-psh function $\f_0$ such that $\sup_X\f_0=0$ and
	\[(\theta+\dc\f_0)^n=cfdV.\]
	We recall that there is $\rho\in\PSH(X,\theta)$ with analytic singularities along a closed subset $E$ such that $\theta+\dc\rho\geq 2\delta\omega_X$ for some $\delta>0$. Set $U:=X\setminus (D\cup E)$.
	\begin{theorem}
		Assume $\f_0$ is as above. Let $\f_t$ be the weak solution of the equation~\eqref{pcmae} with the initial data $\f_0$. Then $\f_t$ converges to $\f_0$ in $\mathcal{C}^\infty_{\rm loc}(U)$.
	\end{theorem}
	\begin{proof}
		The proof is quite close to~\cite[Theorem 6.3]{di2017uniqueness}. We sketch the key steps for the reader's convenience. First, we approximate $\psi^{\pm}$ by their smooth approximants $\psi^{\pm}_j$, thanks to~\cite{demailly1992regularization}. We then apply Tosatti--Weinkove's theorem~\cite{tosatti2010complex} to obtain smooth $(\theta+2^{-j}\omega_X)$-psh functions $\f_{0,j}$ such that $\sup_X\f_{0,j}=0$ and
		\[(\theta+2^{-j}\omega_X+\dc\f_{0,j})^n=c_je^{\psi^+_j-\psi^-_j}dV.\] Note here that $f_j=e^{\psi^+_j-\psi^-_j}$ have uniform $L^p$-norms. The same arguments as in~\cite[Theorem 4.2]{guedj2021quasi} show that 
		\begin{itemize}
			\item $c_j\to c>0$;
			\item for any $\varepsilon>0$,  $\f_{0,j}\geq \varepsilon(\rho+\delta\psi^{-})-C(\varepsilon)$;
			\item $\Delta_{\omega_X}\f_{0,j}\leq e^{-C(\varepsilon)(\rho+\delta\psi^-)}. $
		\end{itemize}
		Let $\f_{t,j}$ be a smooth solution to the equation~\eqref{pcmae} with initial data $\f_{0,j}$. The sequence $\f_{t,j}$ converges to the unique weak solution $\f_t$.
		We apply Proposition~\ref{prop: zero1} and Proposition~\ref{prop: zero2}, together with bootstrapping arguments, to obtain locally uniform estimates for all derivatives of $\f_{t,j}$. This leads to convergence in $\mathcal{C}^\infty_{\rm loc}(U)$.
	\end{proof}
	\section{Finite  time singularities}\label{sect: finite_sing}
	In this section, we study the finite time singularities of the Chern--Ricci flow and provide proof of Theorem~\ref{thmA}.
	
	We consider a family of Hermitian metrics $\omega(t)$ evolving under the Chern--Ricci flow~\eqref{crf} with the initial Hermitian metric $\omega_0$. Suppose that the maximal existence time of the flow is finite, i.e., $T_{\max}<\infty$.  The form $\alpha_{T_{\max}} := \omega_0-T_{\max}\textrm{Ric}(\omega_0)$ is nef in the sense of~\cite{guedj2022quasi2}, i.e., for each $\varepsilon>0$ there exists $\psi_\varepsilon\in\mathcal{C}^\infty(X)$ such that $\alpha_{T_{\max}}  +\dc \psi_\varepsilon\geq -\varepsilon\omega_0$. 
	Indeed, for $\varepsilon>0$,
	\[\alpha_{T_{\max} }  +\varepsilon\omega_0=(1+\varepsilon)\left(\omega_0-\frac{ T_{\max} } {1+\varepsilon}\textrm{Ric}(\omega_0) \right),\]
	and since $\frac{T_{\max}}{1+\varepsilon} <T_{\max}$, we have $\omega_0-\frac{T_{\max}}{1+\varepsilon}\textrm{Ric}(\omega_0)+\dc \psi>0$ for some smooth function $\psi$. We assume that $\alpha_{T_{\max}}$ is {\em uniformly non-collapsing}, i.e.,
	\begin{equation}\label{eq: noncollapsing}
	\int_X(\alpha_{T_{\max} }+\dc\psi)^n\geq c_0>0,   \quad\forall\; \psi\in\PSH(X,\alpha_{T_{\max} })\cap \mathcal{C}^{\infty}(X).
	\end{equation} This condition implies that the volume of $(X,\omega(t))$ does not collapse to zero as $t\to T_{\max }^-$.
	
	\begin{theorem}\label{thm: bigness}
		Let $\alpha$ be a nef (1,1) form satisfying the uniformly non-collapsing condition~\eqref{eq: noncollapsing}. If $X$ admits a Hermitian metric $\omega_X$ such that $v_+(\omega_X)<+\infty$ then $\alpha$ is big.
		
		Conversely, if $\alpha$ is big and $v_-(\omega_X)>0$ then $\alpha$ is uniformly non-collapsing.
	\end{theorem}
	When $\alpha$ is semi-positive or closed the result was proved by Guedj--Lu~\cite[Theorem 4.6, Theorem 4.9]{guedj2022quasi2}, answering the transcendental Grauert--Riemenschneider conjecture~\cite[Conjecture 0.8]{demailly2004numerical}. For our purposes, we would like to extend it when $\alpha$ is no longer closed.
	\begin{proof}
		The proof of this theorem follows the same lines as in~\cite[Theorem 4.6]{guedj2022quasi2}, which is based on ideas from Chiose~\cite{chiose2016kahler}, so we omit it here. 
	\end{proof}
	\begin{remark} When $\omega_0$ is closed, or more generally, is a Guan--Li metric, i.e., $\dc\omega_0=\dc\omega^2_0=0$, the condition~\eqref{eq: noncollapsing} is simply written as $\int_X \alpha^n_{T_{\max} }>0$. 
		The assumption $v_+(\omega_X)< \infty$ or $v_-(\omega_X)>0$ is independent of the choice of the Hermitian $\omega_X$, as shown in~\cite[Proposition 3.2]{guedj2022quasi2}. For additional examples of manifolds where such conditions hold, we refer the reader to~\cite{angella2022plurisigned}.
		
		This result is a slight generalization of~\cite[Theorem 4.3]{nguyen2016complex}, where $\alpha$ is closed semi-positive, and $X$ admits a pluriclosed metric, i.e., $\dc\omega_X=0$. 
	\end{remark}
	As a consequence of Theorem~\ref{thm: bigness}, we give a slight improvement of the main result in~\cite{tosatti2012plurisubharmonic} (see also~\cite[Theorem 4.1]{nguyen2016complex}) which extends the one of Demailly~\cite{demailly1993numerical} to the non-K\"ahler setting.
	\begin{theorem} Let $X$ be a compact complex $n$-manifold equipped with a Hermitian metric $\omega_X$ satisfying $v_+(\omega_X)<\infty$.
		Let $\alpha$ be a nef (1,1) form. Assume that  $x_1,\ldots,x_N\in X$ are fixed points and $\tau_1,\ldots,\tau_N$ are positive constants 
		such that
		\begin{equation*}
		0<\sum_{j=1}^N\tau_j^n< \int_X(\alpha+\dc\psi)^n,\; \forall\, \psi\in \PSH(X,\alpha)\cap\mathcal{C}^\infty(X).
		\end{equation*}
		Then, there exists an $\alpha$-psh function $\f$ with logarithmic poles at $x_1,\ldots,x_N\in X$:
		\[\f(z-x_j)\leq \tau_j\log\|z-x_j\|+O(1)\] in local coordinates near $x_j$, for all $j=1,\ldots,N$.
	\end{theorem}
	\begin{proof}
		By Theorem~\ref{thm: bigness}, we know that $\alpha$ is big. The rest of the proof follows in the same manner as in~\cite[Theorem 1.3]{tosatti2016clalabi}.
	\end{proof}
	
	We go back to the Chern--Ricci flow.
	Observe that one can deduce
	the Chern--Ricci flow~\eqref{crf} to a parabolic complex Monge--Ampère equation
	\begin{equation*}
	\frac{\partial\f_t}{\partial t}=\log\left[ \frac{(\alpha_t+\dc\f_t)^n}{\omega_0^n}\right],\quad \alpha_t+\dc\f>0, \; \f(0)=0
	\end{equation*} where $\alpha_t:=\omega_0-t\textrm{Ric}(\omega_0)$. We assume that the form $\alpha_{T_{\max}}$ is uniformly non-collapsing.
	By Theorem~\ref{thm: bigness},
	there exists a function $\rho$ with analytic singularities such that
	\[\alpha_{T_{\max}}+\dc\rho\geq 2\delta_0\omega_0\] for some $\delta_0>0$. We observe that
	\begin{equation}\label{eq: estimate_max}
	\begin{split}
	\alpha_t+\dc\rho &=\frac{1}{T_{\max}}\left((T_{\max}-t)(\omega_0+\dc\rho)+t(\alpha_{T_{\max}}+\dc\rho) \right)\\
	&\geq \delta_0\omega_0
	\end{split}
	\end{equation} for $t\in[T_{\max}-\varepsilon,T_{\max}]$ with sufficiently small $\varepsilon>0$. Set $$\Omega:=X\setminus \{\rho=-\infty\}.$$
	We establish uniform $\mathcal{C}^\infty_{\rm loc}$ estimates on $\Omega$.
	
	\begin{lemma}\label{lem: bound}
		There is a uniform constant $C_0>0$ such that on $[0,T_{\max})\times X$ we have
		\begin{enumerate}[label=(\roman*)]
			\item $\f\leq C_0$;
			\item $\dot{\f}\leq C_0$;
			\item $\f\geq \rho-C_0$;
			\item $\dot{\f}\geq C_0\rho-C_0$.
		\end{enumerate}
	\end{lemma}
	\begin{proof}
		The proofs of $(i)$ and $(ii)$ follow directly from the classical maximum principle; see e.g.,~\cite[Lemma 4.1]{tosatti2015evolution} or~\cite{tian2006kahler}. 
		
		For $(iii)$, we set $\psi:=\f-\rho$.
		Note that the function $\psi+At\geq -C$ holds on $[0,T_{\max}-\varepsilon]$ with $\varepsilon$ as above. Fix $T_{\max}-\varepsilon<T'<T_{\max}$, assume that $\psi+At$ achieves its minimum at $(t_0,x_0)\in [0,T']\times X$ with $t_0\in (\tmax-\varepsilon,T']$. Note that $x_0\in\Omega$.
		We compute at $(t_0,x_0)$,
		\begin{equation*}
		\begin{split}
		0\geq 	\frac{\partial \psi}{\partial t}+A&=\log\frac{(\alpha_t+\dc\rho+\dc\psi)^n}{\omega_0^n}+A\\
		&\geq \log \frac{(\delta_0\omega_0)^n}{\omega_0^n}+A\geq -C+A
		\end{split}
		\end{equation*} where we have used the estimate~\eqref{eq: estimate_max}. If we choose $A>C$, then we get a contradiction. Thus, we obtain the lower bound for $\psi$, completing the proof.
		
		For $(iv)$, we apply the minimum principle to
		\[Q= \dot{\f}+A\psi+Bt\] where $A$ and $B$ are large constants that will be chosen later. Our goal is to show that $Q\geq -C$ on $X\times[0,\tmax)$. 
		As above, we observe that $Q\geq -C$ on $[0,\tmax-\varepsilon] \times X$. Given any $\tmax-\varepsilon<T'<\tmax$, suppose that $Q$ achieves its minimum on $[0,T']\times X$ at some point $(t_0,x_0)$ with $t_0\in (\tmax-\varepsilon,T']$. Note that $x_0\in \Omega$. At this point, we have
		\begin{align*}
		0\geq \left(\frac{\partial}{\partial t}-\Delta_\omega\right)Q&=-\tr_\omega\textrm{Ric}(\omega_0)+A\dot{\f}-An+A\tr_\omega(\alpha_t+\dc\rho)+B\\
		&\geq \delta_0\tr_\omega\omega_0+A\log\frac{\omega^n}{\omega_0^n}+\tr_\omega\omega_0-An+B
		\end{align*} where $A$ is chosen so large that $$(A-1)(\alpha_t+\dc\rho)-\textrm{Ric}(\omega_0)\geq \omega_0$$ for $t\in[\tmax-\varepsilon,\tmax]$. But since $A\log y-\delta_0 y^{1/n}$ is bounded from above for $y>0$ the arithmetic-geometric inequality yields
		\[\delta_0\tr_\omega\omega_0+A\log\frac{\omega^n}{\omega_0^n}\geq \delta_0\left( \frac{\omega_0^n}{\omega^n}\right)^{1/n}+A\log\frac{\omega^n}{\omega_0^n}\geq -C_1 \] for uniform constant $C_1>0$. If we choose $B=C_1+An$, we obtain
		\[ 0\geq \left(\frac{\partial}{\partial t}-\Delta_\omega\right)Q\geq \tr_\omega\omega_0>0\] which leads to a contradiction. Thus, the desired estimate follows.
	\end{proof}
	\begin{lemma}\label{lem: C2bound}
		There is a uniform constant $C>0$ such that on $ [0,T_{\max})\times X$ we have
		$$\tr_{\omega_0}\omega(t)\leq Ce^{-C\rho}.$$
	\end{lemma}
	\begin{proof}
		Set $\psi=\f-\rho+C_0\geq 0$.
		We apply the maximum principle to \[Q=\log\tr_{\omega_0}\omega-A\psi+e^{-\psi}\]
		where $A>0$ will be determined later.  The idea of using the last term in $Q$ is due to Phong and Sturm~\cite{phong2010dirichlet} and was used in the context of the Chern--Ricci flow in~\cite{tosatti2013chern,tosatti2015evolution,to2018regularizing}. Note that $e^{-\psi}\in (0,1]$.
		
		It suffices to show that $Q$ is uniformly bounded from above. Again, it follows from the definition of $Q$ that $Q\leq C$ on $[0,T_{\max}-\varepsilon]\times X$ for a uniform $C>0$. Fix $\tmax-\varepsilon<T'<T_{\max}$, and suppose that $Q$ achieves its maximum at some point $(t_0,x_0)\in  [0,T']\times X$ with $t\in(\tmax -\varepsilon,T']$. In what follows, we compute at this point. From~\cite[Prop. 3.1]{tosatti2015evolution} (also~\cite[(4.2)]{tosatti2015evolution}) we have
		\[ \left(\frac{\partial}{\partial t}-\Delta_{\omega}\right)\log\tr_{\omega_0}\omega\leq \frac{2}{(\tr_{\omega_0}\omega)^2}\textrm{Re}(g^{\Bar{q}k}(T_0)^p_{kp}\partial_{\Bar{q}}\tr_{\omega_0}\omega)+C\tr_\omega \omega_0,\]
		where $(T_0)^p_{kp}$ denote the torsion terms corresponding to $\omega_0$. At the maximum point $(x_0,t_0)$ of $Q$, we have $\partial_iQ=0$, hence \[\frac{1}{\tr_{\omega_0}\omega}\partial_i\tr_{\omega_0}\omega-A\partial_i\psi-e^{-\psi}\partial_i\psi=0.\]
		Thus, the Cauchy-Schwarz inequality yields
		\begin{equation*}
		\begin{split}
		\left| \frac{2}{(\tr_{\omega_0}\omega)^2}\textrm{Re}(g^{\Bar{q}k}(T_0)^p_{kp}\partial_{\Bar{q}}\tr_{\omega_0}\omega)\right|&\leq \left| \frac{2}{(\tr_{\omega_0}\omega)^2}\textrm{Re}((A+e^{-\psi})g^{\Bar{q}k}(T_0)^p_{kp}\partial_{\Bar{q}}\psi\right|\\
		&\leq e^{-\psi}|\partial\psi|^2_\omega+C(A+1)^2e^\psi\frac{\tr_\omega \omega_0}{(\tr_{\omega_0}\omega)^2}.
		\end{split}
		\end{equation*} for uniform $C>0$ only depending on the torsion term.
		It thus follows that, at the point $(t_0,x_0)$,
		\begin{equation}\label{eq: c2Q}
		\begin{split}
		0\leq \left(\frac{\partial}{\partial t}-\Delta_{\omega}\right)Q &\leq C(A+1)^2e^\psi\frac{\tr_\omega \omega_0}{(\tr_{\omega_0}\omega)^2}+C\tr_\omega \omega_0\\
		&\quad -(A+e^{-\psi})\dot{\f} +(A+e^{-\psi})\tr_\omega(\omega -(\alpha_t+\dc\rho))\\
		&\leq C(A+1)^2e^\psi \frac{\tr_\omega \omega_0}{(\tr_{\omega_0}\omega)^2}+(C-A\delta_0)\tr_\omega \omega_0+(A+1)\log\frac{\omega_0^n}{\omega^n}
		\end{split}
		\end{equation} where we have used $\alpha_t+\dc\rho\geq \delta_0\omega_0$. If at $(x_0,t_0)$,  $(\tr_{\omega_0}   \omega)^2\leq e^\psi C(A+1)^2$  then at the same point we obtain \[Q\leq C+\frac{1}{2}\psi-A\psi+e^{-\psi}\leq C+1\] since $\psi\geq 0$, we are done.
		Otherwise, we choose $A=\delta_0^{-1}(C+2)$. Hence, from~\eqref{eq: c2Q} one gets
		\[\tr_{\omega}\omega_0\leq C\log\frac{\omega_0^n}{\omega^n}+C.\]
		By Lemma~\ref{lem34}, we obtain
		\[ \tr_{\omega_0}   \omega\leq n(\tr_\omega\omega_0)^{n-1}\frac{\omega^n}{\omega_0^n}\leq C\frac{\omega^n}{\omega_0^n}\left(\log\frac{\omega_0^n}{\omega^n}\right)^{n-1}+C\leq C'\]since $\omega^n/\omega_0^n\leq C_0$ by Lemma~\ref{lem: bound}, and $y\mapsto y|\log y|^{n-1}$ is bounded from above as $y\to 0$. Thanks to Lemma~\ref{lem: bound} (iii), $Q$ is bounded from above at its maximum, finishing the proof.
	\end{proof}
	\begin{proof}[Proof of Theorem~\ref{thmA}] 
		Let $K\subset \Omega$ be any compact set. It follows from Lemma~\ref{lem: bound} and Lemma~\ref{lem: C2bound} that there exists a constant $C_K>0$ such that on $[0,\tmax)\times K$, \[C_K^{-1}\omega_0\leq\omega(t)\leq C_K\omega_0.\]
		Applying the local higher-order estimates of Gill~\cite[Section 4]{gill2011convergence}, we obtain uniform $\mathcal{C}^\infty$ estimates for $\omega(t)$ on compact subsets of $\Omega$. 
		Consequently, there exists a constant $c_K$ such that $\frac{\partial}{\partial t}\omega=-\textrm{Ric}(\omega)\leq c_K\omega$, on $[0,\tmax) \times K$. This implies that $e^{-c_Kt}\omega(t)$ decreases in $t$ and is bounded from below. Hence, $\omega(t)$ converges to $\omega_{\tmax}$ as $t\to \tmax$, and since we have uniform estimates on compact subsets of $\Omega$, we see that the convergence is in $\mathcal{C}_{\rm loc}^\infty(\Omega)$.
		This finishes the proof. 
	\end{proof}
	
	\section{The Chern--Ricci flow on varieties with log terminal singularities}\label{sect: crf_lt}
	In this section, we extend our previous analysis to the case of compact complex varieties with {\em mild singularities}. We refer the reader to~\cite[Section 5]{eyssidieux2009singular} for a brief introduction to the complex analysis on mildly singular varieties. 
	

	We assume here that $Y$ is a $\mathbb{Q}$-Gorenstein variety, i.e., $Y$ is a normal complex space such that its canonical divisor $K_Y$ is $\mathbb{Q}$-Cartier.
	We denote the singular set of $Y$ by $Y_{\rm sing}$ and let $Y_{\rm reg}:=Y\setminus Y_{\rm sing}$.
	Given a log resolution of singularities $\pi:X \to Y$ (which may and will always be chosen to be an isomorphism over $Y_{\rm reg}$ ), there exists  a unique (exceptional) $\mathbb{Q}$-divisor $\sum a_iE_i$ with simple normal  crossings (snc for short) such that
	\begin{equation*}
	K_X=\pi^*K_Y+\sum_ia_iE_i,
	\end{equation*} The coefficients $a_i\in\mathbb{Q}$ are called the {\em discrepancies} of $Y$ along $E_i$. 
	\begin{definition}
		We say that $Y$ has \emph{log terminal} (\emph{lt} for short) singularities if and only if $a_i>-1$ for all $i$.  
	\end{definition}
	
	The following definition of \emph{adapted measure} is introduced in \cite[Section 6]{eyssidieux2009singular}:
	\begin{definition}
		Let $h$ be a smooth hermitian metric on the $\mathbb{Q}$-line bundle $\mathcal{O}_Y (K_Y)$. The corresponding adapted measure $\mu_{Y,h}$ on $Y_{\rm reg}$ is locally defined by choosing a nowhere vanishing  section $\sigma$ of $mK_Y$ over a small open set $U$ and setting
		\begin{equation*}
		\mu_{Y,h}:=\frac{(i^{mn^2}\sigma\wedge\Bar{\sigma})^{1/m}}{|\sigma|_{h^m}^{2/m}}.
		\end{equation*}
	\end{definition}  
	
	The point of the definition is that the measure $\mu_{Y,h}$ does not depend on the choice of $\sigma$, so it is globally defined. The arguments above show that $Y$ has log terminal singularities if and only if  $\mu_{Y,h}$ has a finite total mass on $Y$, which can be considered as a Radon measure on the whole of $Y$. Then $\chi=\dc\log\mu_{Y,h}$ is a well-defined smooth closed $(1,1)$-form on $Y$.

	Given a Hermitian form $\omega_Y$ on $Y$, there exists a unique hermitian metric $h=h(\omega_Y)$ of $K_Y$ such that 
	\[\omega_Y^n=\mu_{Y,h}.\]
	We have the following definition.
	\begin{definition}
		The \emph{Ricci curvature form} of $\omega_Y$ is $\textrm{Ric}(\omega_Y):=-\dc\log h$.
	\end{definition}
	We recall the {\em slope} of a quasi-psh function $\phi$ at $y$ in the sense of~\cite{berman2019kahler}. Choosing local generators $(f_j)$ of the maximal ideal $\mathfrak{m}_y$ of $\mathcal{O}_{Y,y}$, we define
	\[s(\phi,y)=\sup\{s\geq 0: \f\leq s\log\sum|f_j|+O(1)\}.\]
	Note that this definition is independent of the choice of $(f_j)$. By \cite[Theorem A.2]{berman2019kahler} there is $C>0$ such that for any log resolution $\pi:X\to Y$, \[  \nu(\phi\circ \pi,E)\leq Cs(\phi,y)\]
	with $E$ a prime divisor lying above $y$. In particular, the Lelong numbers of $\phi\circ\pi$ are sufficiently small if the $s(\phi,y)$ is also sufficiently small at all points $y\in Y$. We refer to~\cite{pan2025-Lelong-numbers} for related results.

	\medskip
	Applying the analysis in the previous section, we obtain the existence of the Chern–Ricci flow on log terminal singularities. This generalizes the result in \cite[Theorem E]{dang2021chern}.
	\begin{theorem}
		Let $Y$ be a compact complex variety with log terminal singularities. Assume that $\theta_0$ is a Hermitian metric such that 
		\begin{equation*}
		T_{\rm max}:=\sup \{t>0:\, \exists\; \psi\in\mathcal{C}^\infty(Y) \,\text{such that}\; \theta_0-t\textrm{Ric}(\theta_0)+\dc\psi >0 \}>0.
		\end{equation*} Assume that $S_0=\theta_0+\dc\phi_0$ is a positive (1,1)-current with small slopes. Then, there exists a family $(\omega_t)_{t\in[0,\tmax)}$ of positive (1,1) current  on $Y$ starting with $S_0$ such that
		\begin{enumerate}
			\item $\omega_t=\theta_0-t\textrm{Ric}(\theta_0)+\dc\f_t$ are positive (1,1) currents;
			\item $\omega_t\to S_0$ weakly as $t\to 0^+$;
			\item for each $\varepsilon>0$, there exists a Zariski open set $\Omega_\varepsilon$ such that on $[\varepsilon,\tmax)\times \Omega_\varepsilon$,  $\omega$ is smooth and
			\[\frac{\partial\omega}{\partial t}=-\textrm{Ric}(\omega).\]
		\end{enumerate}
	\end{theorem}
	\begin{proof}
		It is classical that solving the (weak) Chern--Ricci flow is equivalent to solving a complex Monge--Amp\`ere equation flow. 
		Let $\chi$ be a closed smooth (1,1) form that represents $c^{\rm BC}_1(K_Y)$. Given $T\in (0,T_{\rm max})$, there is a function $\psi_T\in\mathcal{C}^\infty(Y)$ such that $\theta_0-t\textrm{Ric}(\theta_0)+\dc\psi_T >0$.
		We set for $t\in[0,T]$
		\begin{equation*}
		\hat{\theta}_t:=\theta_0+t\chi,\; \text{with}\, \chi= -\textrm{Ric}(\theta_0)+\dc\frac{\psi_T}{T},
		\end{equation*} which defines an affine path of Hermitian forms. Since $\chi$ is a smooth representative of $c_1^{\rm BC}(K_Y)$, one can find a smooth metric $h$ on the $\mathbb{Q}$-line bundle $\mathcal{O}_Y(K_Y)$ with curvature form $\chi$. We obtain $\mu_{Y,h}$, the adapted measure corresponding to $h$.
		The Chern-Ricci flow is equivalent to the following complex Monge-Amp\`ere flow 
		\begin{equation}\label{eq: maf}  (\hat{\theta}_t+\dc\phi_t)^n=e^{\partial_t{\phi}}\mu_{Y,h}.
		\end{equation} 
		Let $\pi:X \to Y$ be a log resolution of singularities. We have seen that the measure 
		$$\mu:=\pi^*\mu_{Y,h}=fdV \quad\text{where}\;\; f=\prod_{i}|s_i|^{2a_i}$$ has  poles (corresponding to $a_i<0$) or zeroes (corresponding to $a_i>0$) along the exceptional divisors $E_i=(s_i=0)$, and $dV$ is a smooth volume form. Passing to the resolution, the flow~\eqref{eq: maf} becomes
		\begin{equation}\label{eq: cmaf}
		\frac{\partial \f}{\partial t}=\log\left[ \frac{(\theta_t+\dc\f_t)^n}{\mu}\right]
		\end{equation} where $\theta_t:=\pi^*\hat{\theta}_t$ and $\f:=\pi^*\phi$. Since $(\hat{\theta}_t)_{t\in [0,T]}$ is a smooth family of Hermitian forms, it follows that the family of semi-positive forms $[0,T]\ni t\mapsto\theta_t$ satisfies all our requirements. We also have that  ${\theta}:=\pi^*\theta_0$, the latter is smooth, semi-positive, and big but no longer hermitian. We fix a ${\theta}$-psh function ${\rho}$ with analytic singularities along a divisor $E=\pi^{-1}(Y_{\rm sing})$ such that $\theta+\dc{\rho}\geq 2\delta\omega_X$ with $\delta>0$. 
		If we set $\psi^{+}=\sum_{a_i>0}2a_i\log|s_i|$, $\psi^{-}=\sum_{a_i<0}-2a_i\log|s_i|$, we observe that $\psi^{\pm}$ are quasi-psh functions with logarithmic poles along the exceptional divisors, smooth on $X\setminus \textrm{Exc}(\pi)=\pi^{-1}(Y_{\rm reg})$, and $e^{-\psi^-}\in L^p(dV)$ for some $p>1$. 
		We observe that since the Lelong numbers $\nu(\f_0,x)$ are sufficiently small, we have the assumption $p^*/2c(\f_0)<\tmax$ by Skoda's integrability theorem.
		The result therefore follows from Theorem~\ref{thmB}. 
	\end{proof}
	\bibliographystyle{plain}
	\bibliography{bibfile}	
	
\end{document}